\documentclass[english,11pt,reqno,twoside]{amsart}

\usepackage[english]{babel}
\usepackage[left=2.4cm,right=2.4cm,top=2cm,bottom=2cm]{geometry}

\usepackage{amsmath,amssymb,amsthm,amsfonts}
\usepackage{color}

\usepackage{amsfonts,color}
\usepackage{latexsym,amssymb}
\usepackage{amsmath}
\usepackage{enumitem}
\usepackage[colorlinks=true,urlcolor=blue,
citecolor=red,linkcolor=blue,linktocpage,pdfpagelabels,
bookmarksnumbered,bookmarksopen]{hyperref}
\usepackage{mathrsfs}
\usepackage[hyperpageref]{backref}

\usepackage[english]{babel}
\usepackage[utf8]{inputenc}
\usepackage{tikz}
\usetikzlibrary{patterns}

\newtheorem{theorem}{Theorem}[section]
\newtheorem{proposition}[theorem]{Proposition}

\newtheorem{lemma}[theorem]{Lemma}
\newtheorem{definition}[theorem]{Definition}
\newtheorem{remark}[theorem]{Remark}

\title[Existence and nonexistence for
$(p,q)$-Laplacian operator]{On a zero-mass $(p,q)$-Laplacian equation involving subcritical and supercritical growth in $\mathbb{R}^N$}

\author[E. Souto de Medeiros]{Everaldo Souto de Medeiros}
\author[A. Ara\'ujo do Nascimento]{Antonio Ara\'ujo do Nascimento}

\author[G. Siciliano]{Gaetano Siciliano}

\address[A. Ar\'aujo do Nascimento]{\newline\indent Departamento de Matem\'atica
\newline\indent 
Universidade Federal da Para\'iba,
\newline\indent 
Jo\~ao Pessoa PB, 58051-900, Brazil}
\email{\href{mailto:aadn@estudantes.ufpb.br}{aadn@estudantes.ufpb.br}}

\address[E. Souto de Medeiros]{\newline\indent Departamento de Matem\'atica
\newline\indent 
Universidade Federal da Para\'iba,
\newline\indent 
Jo\~ao Pessoa PB, 58051-900, Brazil}
\email{\href{mailto:everaldo@mat.ufpb.br}{everaldo@mat.ufpb.br}
}

\address[G. Siciliano]{\newline\indent Dipartimento di Matematica
\newline\indent 
 Universit\`a degli Studi di Bari Aldo Moro
\newline\indent 
Via E. Orabona 4, 70125 Bari, Italy}
\email{\href{mailto:gaetano.siciliano@uniba.it}{gaetano.siciliano@uniba.it}}



\begin{document}

\begin{abstract}
This paper is concerned with the zero-mass $(p,q)$-Laplacian equation
$$
-\Delta_p u-\Delta_q u = |u|^{r-2}u+\lambda |u|^{s-2}u
\quad \text{in } \mathbb{R}^N,
$$
where $1<q<p<N$. The exponents $r$ and $s$ may belong to either the subcritical or the supercritical range with respect to the critical Sobolev exponents $p^*$ and $q^*$. We establish existence and nonexistence results and show that the sign of the parameter $\lambda$ determines the solvability regimes of the equation. The existence proofs rely on variational methods, truncation arguments, and regularity theory, while the nonexistence results are derived from a suitable Pohozaev-type identity for the $(p,q)$-Laplacian operator.
\end{abstract}

\maketitle

\begin{center}
\begin{minipage}{12cm}
\tableofcontents
\end{minipage}
\end{center}

\medskip

\section{Introduction and statement of the results}
In this paper, we study the existence and nonexistence of solutions to a nonlinear elliptic equation driven by the nonhomogeneous $(p,q)$-Laplace operator of the form
\begin{equation} \label{P}
-\Delta_p u-\Delta_q u=|u|^{r-2}u+\lambda|u|^{s-2}u\quad\mbox{in} \quad \mathbb{R}^N,
\tag{$\mathcal{P}_\lambda$}
\end{equation}
where $1<q<p<N$, $\lambda$ is a real parameter and the exponents $r,s$ satisfy
$$
\frac{Nq}{N-q}=q^* < r,s < p^*=\frac{Np}{N-p}\quad\mbox{or}\quad p^*<s.
$$
Here, for $1<m<\infty$, $-\Delta_m u:=-\rm{div}(|\nabla u|^{m-2}\nabla u )$ denotes the classical  $m-$Laplacian operator.

Problems involving $(p,q)$-Laplacian type operators have attracted considerable attention in recent years due to both their rich mathematical structure and their broad range of applications in physics and related sciences:
they appear in modeling of nonlinear reaction-diffusion processes, non-Newtonian fluids, plasma physics, biophysics, nonlinear elasticity, and materials exhibiting nonstandard growth behavior (see e.g. \cite{AC95}). In particular, such models describe media in which different diffusion mechanisms interact and compete according to the intensity of the gradient.

From a mathematical perspective, elliptic equations involving such operators have attracted considerable attention in the literature, including in the context of critical growth problems. This interest is largely motivated by the fact that the interaction between different values of $p$ and $q$ induces a nonhomogeneous variational structure, resulting in significant analytical difficulties concerning compactness, regularity, and qualitative properties of solutions.

In bounded domains, such equations have been extensively studied by many authors. For instance, in \cite{LZ09}, the authors consider the critical problem and establish the existence of infinitely many solutions. In \cite{MP09}, the problem is investigated using cohomological index theory. Additional existence and multiplicity results can be found, for example, in \cite{CI05,DO96}.

In the whole space $\mathbb{R}^N$,  elliptic equations involving the $(p,q)$-Laplacian operator have also been extensively studied by several authors; see, for instance, \cite{AR24, EPW25}. In the aforementioned works, the authors consider elliptic equations in the positive-mass setting, that is, equations of the form
$$
-\Delta_p u - \Delta_q u + V(x)(u^{p-1} + u^{q-1}) = f(x,u) \quad \text{in } \mathbb{R}^N,
$$
where the potential $V$ satisfies suitable assumptions that allow one to recover compactness, while the nonlinearity $f(x,u)$ may exhibit subcritical, critical, or supercritical growth.

In the zero-mass case, that is, when $V\equiv 0$, the equation becomes considerably more challenging due to the lack of compactness. In this case, we mention the papers \cite{BF23, BBF22, BBF21, LP24}, where the authors considered an equation of the form
\begin{equation}\label{eq:zero-mass}
-\Delta_p u - \Delta_q u = \lambda W(x)|u|^{q-2}u + K(x)|u|^{p^{*}-2}u \quad \text{in } \mathbb{R}^N,
\end{equation}
with $1<q<p<N$, under suitable assumptions on the weight functions $W$ and $K$, the authors established existence and multiplicity results. In particular, the function $W$ is required to be nonconstant, and therefore the techniques developed in those works cannot be applied to our setting. Among recent contributions in the zero-mass setting, we also mention \cite{CFFM21}, where the authors considered the equation
$$
-\Delta_N u-\Delta_q u=f(u)
\quad \text{in } \mathbb{R}^N,
$$
with a nonlinearity exhibiting critical exponential growth in the Trudinger--Moser sense.

Other results  concerning existence, multiplicity, normalized solutions, critical growth, and jumping nonlinearities for $(p,q)$-Laplacian equations, even in the fractional case, can be found for instance, in \cite{AIS19, AOS20, DQOL26, RGP25}.

Finally we cite the paper 
\cite{HL08} where the regularity of weak solutions
for the $(p,q)$-Laplacian equation has been addressed.

\medskip

Before stating our main results, we introduce some notation and definitions. For \(m\in(1,N)\), let 
$$
\mathcal{D}^{1,m}(\mathbb{R}^N)=\left\{u\in L^{m^*}(\mathbb R^N): |\nabla u|\in L^m(\mathbb R^N) \right\}
$$
the homogeneous Sobolev space which is a reflexive Banach space when endowed with the norm
$$
\|u\|_{\mathcal{D}^{1,m}}: =
\|\nabla u\|_m^m=\left(\int_{\mathbb R^{N}}|\nabla u|^m \right)^{1/m}.
$$
Consequently, the space
$E=:\mathcal{D}^{1,p}(\mathbb{R}^N)\cap \mathcal{D}^{1,q}(\mathbb{R}^N)$,
endowed with the norm
$$\| u\|_E =\Big(\|u\|^p_{\mathcal D^{1,p}} + \|u\|_{\mathcal D^{1,q}}^{p/q}\Big)^{1/p}$$
is a reflexive Banach space. Furthermore, the classical Gagliardo–Nirenberg–Sobolev inequality implies that the embedding 
\begin{equation}\label{Principal-Embedding}
E\hookrightarrow L^t(\mathbb{R}^N), \quad\mbox{for any}\quad q^*\leq t\leq p^*
\end{equation}
is continuous.

Here, by a weak solution of equation \eqref{P} we mean a function $u \in E$ such that for any $\varphi \in E$ it holds
\begin{equation} \label{idweak}
\int_{\mathbb{R}^N}\left(|\nabla u|^{p-2}\nabla u+|\nabla u|^{q-2}\nabla u\right)\nabla\varphi=\int_{\mathbb{R}^N}\left(|u|^{r-2} u+\lambda|u|^{s-2} u\right)\varphi. 
\end{equation}

In this paper, our aim is to establish existence and nonexistence results depending on the sign and magnitude of the parameter $\lambda$, and on whether the exponents $r,s$ are subcritical or supercritical relative to the critical exponents $q^*$ and $p^*$, highlighting the interplay between these features.
\medskip

We first consider the case where the real parameter $\lambda$ is positive. Our main existence results for equation \eqref{P} are stated below.\begin{theorem}
[$\lambda>0$, subcritical case]\label{T1} 
Suppose $\lambda>0$ and that $q^*< r,s< p^*$. Then the following assertions hold:
\begin{enumerate}[label=(\alph*), ref=\alph*]
    \item \label{pmenor}  if $p<\min\{r,s\}$, then for any $\lambda>0$, the equation \eqref{P} has a positive weak  solution, \smallskip

    \item \label{pnomeio}if $s<p<r$, then there exists a sufficiently small $\lambda_0 > 0$ such that for every $\lambda \in (0, \lambda_0)$, the equation \eqref{P} admits a positive weak solution. 
\end{enumerate}
\end{theorem}
\begin{theorem}
[$\lambda>0$, supercritical case] 
\label{th:supercritical}
Suppose $\lambda>0$ and that $q^* < r < p^* < s$. If $p<r$, then there exists a sufficiently small $\lambda_0 > 0$ such that for any $\lambda \in (0, \lambda_0)$, equation \eqref{P} admits a positive weak solution. 
\end{theorem}
When the parameter $\lambda$ is negative, our main result is stated as follows.
\begin{theorem}[$\lambda<0$, subcritical case] \label{T2}
Assume that $q^* < r,s < p^*$. Then the following assertions hold:
\begin{enumerate}[label=(\alph*), ref=\alph*]
    \item \label{rmaior} if $r>\max\{s,p\}$, then for any $\lambda<0$, the equation \eqref{P} has a positive weak solution, \smallskip
    \item \label {rnomeio} if $p<r<s$, then there exists $\lambda_0 < 0$ sufficiently close to $0$ such that, for any $\lambda \in (\lambda_0, 0)$ equation \eqref{P} has a positive weak solution.
\end{enumerate}
\end{theorem}

As we will see in Section~\ref{sec:RegPosit}, by using standard arguments, the solutions obtained in all of the previous cases are in fact of class $C^{1,\alpha}_{\mathrm{loc}}(\mathbb{R}^N)$ for some $\alpha\in(0,1)$.\\

Finally we present also a nonexsistence result.
\begin{theorem}\label{th:NE}
There exists no nontrivial weak solution of problem
\eqref{P} in $E\cap W^{2,p}_{\emph{loc}}(\mathbb R^N)$ under the following conditions on the parameter $\lambda$ and the powers $q,p,r,s$:
\begin{enumerate}
    \item[(a)] if $\lambda > 0$ and
    \[
        p^* \leq \min\{r,s\}
        \quad \text{or} \quad
        \max\{r,s\} \leq q^*;
    \]

    \item[(b)] if $\lambda < 0$ and
    \[
        s \leq p^* \leq r
        \quad \text{or} \quad
        r \leq q^* \leq s.
    \]
\end{enumerate}
\end{theorem}

\begin{remark} \rm
For the case $\lambda=0$ by using the same approach employed in Theorem \ref{T1} \eqref{pmenor}, 
one can obtain the existence of a nontrivial solution
provided that
$p<r$. Moreover, the nonexistence result holds if
$r\leq q^*$ or $p^*\leq r$.
\end{remark}

\begin{remark} \rm
The limit cases $r=q^{*}$ or $s=q^{*}$ are not covered by the arguments used here. Indeed, the non vanishing of the Palais-Smale sequence may occur  in the $L^{q^{*}}(\mathbb{R}^N)$-norm, preventing the direct application of Lemma \ref{Lions-Lemma}.
\end{remark}

Since the problem is variational, the solutions will be found as critical points of a suitable energy functional. 
The main difficulties in proving our existence results are due to the different degree of homogeneity of the nonlinearity which causes a different behavior in the geometry of the functional when interacting with the $(p,q)$-Laplacian. We remark also that no potentials are present in the equation so a further difficulty is related to obtaining compactness. To overcome these difficulties suitable truncation arguments are implemented as well Lions-type concentration-compactness techniques.

\medskip
The paper is organized as follows. In Section \ref{sec:prelim} we give some preliminaries which are useful to implement variational methods. Then Section \ref{sec:Th1.2} is devoted to prove the existence results, Theorem \ref{T1}, 
Theorem \ref{th:supercritical} and Theorem \ref{T2}. Section~\ref{sec:RegPosit} is devoted to proving the positivity and $C^{1, \alpha}_\textrm{loc}(\mathbb R^N)$ regularity of the solutions obtained. Finally in Section \ref{sec:NE} we prove Theorem \ref{th:NE} after establishing a Pohozaev identity.

\medskip
\noindent {\bf Notations.}
We introduce few basic notations and conventions used  through the paper.

\begin{itemize}
\item $B_R(y)$ will denote the ball in $\mathbb R^N$ centered in $y$ and with radius $R>0$;
\item if $A\subset \mathbb R^N$ is measurable, its Lebesgue measure is denoted with $|A|$;
\item in the integrals the Lebesgue measure $dx$ is always omitted;
\item $C_r,C_s>0$  denote the embedding constants of $E$ into $L^r(\mathbb R^N)$ and $L^s(\mathbb R^N)$ respectively;
\item the norm in $L^p$ is denoted with $\| \cdot \|_p$ whenever the domain is the whole space
and with $\| \cdot \|_{L^p(A)}$ if $A\neq \mathbb R^N$;
\item $o_n(1)$ is a generic vanishing sequence;
\item $u_+ = \max\{u,0\}$ and $u_- = - \min\{u,0\}$.
\end{itemize}

Other notations will be introduced whenever we need.

\section{Preliminary Results}\label{sec:prelim}
In this Section we give useful results that we will use through the paper.

First of all observe  that the norm $\|u\|_E$ defined in the previous Section and $\|u\|:= \|u\|_{\mathcal D^{1,p}} + \|u\|_{\mathcal D^{1,q}}$ are equivalent on $E$.
Indeed, for any $a,b$ nonnegative real numbers, let us say $a\leq b$, we have
$$
a+b\leq 2b\leq2(a^p+b^p)^{1/p}\leq2^{1+\frac{1}{p}}b\leq2^{1+\frac{1}{p}}(a+b).
$$
Taking $a=\left(\displaystyle\int_{\mathbb R^N}|\nabla u|^p\right)^{1/p}$ and $b=\left(\displaystyle\int_{\mathbb R^N}|\nabla u|^q\right)^{1/q}$ in the above inequalities we obtain the equivalence. Moreover from Young's inequality, we have, for $t>1$, 
$$
\|\nabla u\|_{t} \leq\frac{1}{t} \|\nabla u\|_{t}^t+\frac{t-1}{t}\leq \|\nabla u\|_{t}^t+1.
$$
Hence by adding the two inequalities obtained  for $t=p$ and $t=q$ we get
the usefull inequalities
\begin{equation}\label{eq:usefull}
  \kappa  \|u\|_E \le \| \nabla u\|_p +\| \nabla u\|_q \le
    \| \nabla u\|_p^p + \| \nabla u\|_q^q+2, \quad \kappa = 2^{-1/p}.
\end{equation}

The following result is standard but fundamental for our purpose. We give the details for the reader convenience.
\begin{proposition}\label{Densidade}
  The space  $C_0^\infty(\mathbb{R}^N)$ is dense in $E$.
\end{proposition}
\begin{proof}
Let $u\in E$ and consider  $\eta_R\in C_0^\infty(\mathbb{R}^N)$ such that 
$$
0\le \eta_R\le1,\quad \eta_R=1 \ \text{em } B_R(0),\quad  \eta_R=0 \ \text{em } \mathbb{R}^N \setminus B_{2R}(0)\quad \text{and} \quad  |\nabla \eta_R|\le \frac{2}{R}.
$$
Define
\[
u_R:=\eta_R u \in E
\]
which has compact support and 
$\nabla u_R=\eta_R\nabla u+u\nabla\eta_R.$
Thus, 
\[
\nabla(u_R-u)=(\eta_R-1)\nabla u+u\nabla\eta_R.
\]
Let us estimate the two terms above.
We have
\begin{equation}\label{eq:primeira}
    \|(\eta_R-1)\nabla u\|_{p}
=\|\nabla u\|_{L^p(\mathbb{R}^N\setminus B_R(0))}\to0\quad\mbox{as } \  R\to \infty.
\end{equation}

For the second term, we observe that $\nabla\eta_R$ is supported on the ring $A_R:=B_{2R}(0)\setminus B_R(0)$ and by the  Sobolev inequality on bounded domains, we obtain
\begin{equation}\label{eq:segunda}
\|u\nabla\eta_R\|_{p} \leq \frac{2}{R}\|u\|_{L^p(A_R)} \leq \frac{C}{R}\|\nabla u\|_{p}\to 0\quad \text{as } \ R\to\infty.
\end{equation}
Thus, since \eqref{eq:primeira} and \eqref{eq:segunda} also follow in the $L^q-$norm, we get
\[
\|u-u_R\|= \|\nabla(u_R-u)\|_{p}+\|\nabla(u_R-u)\|_{q}\to0 \quad \text{as} \quad
R\to\infty.
\]
that is, $u_R\to $ in $E$. Now, given $\delta>0$, consider $\rho_\delta$ a standard mollifier function and define
$$
u_{R,\delta}:=\rho_\delta*u_R \in C_0^\infty(\mathbb{R}^N).
$$
It holds, for any fixed $R>0$,
$
\nabla u_{R,\delta}
=\rho_\delta*\nabla u_R
\longrightarrow\nabla u_R
$
in $L^p(\mathbb R^N)\cap  L^q(\mathbb R^N)$,
as $\delta \to 0$, namely
$$\|u_R - u_{R,\delta}\|\to 0 \quad \text{as} \quad \delta\to 0.$$
Then given $\varepsilon>0$ fix $R=R_\varepsilon>0$ such that 
$\| u - u_{R}\| \le \varepsilon$ and choose $\delta = \delta_\varepsilon>0$ such that $\| u_{R}- u_{R , \delta}\| \le \varepsilon$. Thus, we infer that
$\| u - u_{R , \delta}\| \le 2\varepsilon$
proving the result.
\end{proof}

It is useful also the following fact.
If $\|u\|_E<1$, then it is also
    $\displaystyle\int_{\mathbb{R}^N} |\nabla u|^q < 1$ which implies, being $q<p$, 
    \begin{equation}\label{eq:estimativaE}
      \frac{1}{p}\int_{\mathbb R^N} |\nabla u|^p+
      \frac{1}{q}\int_{\mathbb R^N} |\nabla u|^q\ge \frac{1}{p}\int_{\mathbb R^N} |\nabla u|^p +\frac{1}{q}
       \left(\int_{\mathbb{R}^N} |\nabla u|^q \right)^{\frac{p}{q}}
    \ge
     \frac{1}{p} \|u\|_E^p
    \end{equation}
This inequality will be used many times through  the paper.

One of the essential ingredients in our work is the following Lions type vanishing lemma.

 \begin{lemma}\label{Lions-Lemma}
Let $\{u_n\} $ be a bounded sequence 
in $L^{q^*}(\mathbb R^N)\cap L^{p^*}(\mathbb R^N) $
and suppose that for some $R>0$ and $t\in[q^*,p^*)$ it holds
 	$$
\sup_{y\in\mathbb{R}^N}\int_{B_R(y)}|u_n|^tdx\rightarrow 0, \quad \mbox{as}\quad n\rightarrow+\infty.
 	$$ 
 	Then, $u_n\rightarrow 0$ in $L^\tau(\mathbb{R}^N)$ for all $q^*<\tau<p^*$.
 \end{lemma}
 \begin{proof} For $y\in \mathbb R^N$ and $\rho>0$
 consider the cube centered in $y$ and side $2\rho$,  $Q(y,r)\subset B_R(y)$.
 Then for any $\tau\in[t, p^*)$ 
 by interpolation we have
 	\begin{eqnarray*}
\|u_n\|^\tau_{L^\tau(Q(y,\rho))}&\leq&\|u_n\|_{L^t(B_R(y))}^{(1-\alpha)\tau}\|u_n\|_{L^{p^*}(Q(y,\rho))}^{\alpha \tau} \\
&\le & \Big(\sup_{y\in\mathbb{R}^N}\int_{B_R(y)}|u_n|^t \Big)^{\frac{(1-\alpha)\tau}{t}}\|u_n\|_{L^{p^*}(Q(y,\rho))}^{\alpha \tau}
 	\end{eqnarray*}
 	for a suitable $\alpha \in [0,1)$. 
Then by choosing $\{y_{k}\}$ such that
 $$\mathbb R^{N}= \bigcup_{k} \overline{Q}(y_{k}, \rho) \quad\text{and}\quad Q(y_{k}, \rho)\cap Q(y_{l}, \rho)=\emptyset,$$
 we have
 \begin{eqnarray*}
 \|u\|^{\tau}_{\tau} \leq  \Big( \sup_{y\in \mathbb R^{N}}  \int_{B_R(y)}|u_n|^{t} \Big)^{\frac{(1-\lambda)\tau}{t}}\|u_n\|^{\lambda \tau}
 _{p^*}\to 0.    
 \end{eqnarray*}
If now $\tau\in (q^*, t)$ by interpolation, for suitable $\alpha\in(0,1)$
$$\| u_n\|_\tau \le \| u_n\|^{1-\alpha}_{q^*}\| u_n\|^{\alpha}_t \to 0,
$$
which concludes the proof
 \end{proof}

As we said before, the problem is variational and a suitable energy functional can be defined in the space $E$ in such a  way that the weak solutions of \eqref{P} can be characterized as its critical points.
Using the continuous embedding \eqref{Principal-Embedding} and a straightforward computation, it follows that the functional $I_\lambda: E\rightarrow\mathbb{R}$ 
defined by,
 $$
 I_\lambda(u):=\frac{1}{p}\int_{\mathbb{R}^N}|\nabla u|^p+\frac{1}{q}\int_{\mathbb{R}^N}|\nabla u|^q-\frac{1}{r}\int_{\mathbb{R}^N}u^r_+ -\frac{\lambda}{s}\int_{\mathbb{R}^N}u^s_+, \quad u\in E,
 $$
at least for $q^{*} \leq r,s \leq p^{*}$, is well defined and is of class $C^1$. Furthermore,  critical points of $I_\lambda$  give exactly positive solutions of \eqref{P}.

\section{Existence: proof of the main results}
\label{sec:Th1.2}
In this section, we prove our existence results for equation \eqref{P} according to the values of the exponents $s, r$ and the parameter $\lambda$.

The results are achieved via a careful analysis of the Mountain Pass structure of the functional.

\subsection{Proof of Theorem \ref{T1} (\ref{pmenor})}

The next result deals with the Mountain Pass Geometry for the functional $I_\lambda$.

\begin{lemma}
\label{MP} 
Assume that $1<q<p<N$ and $q^*\leq r,s\leq p^*$ and $p < \min\{r,s\}$. There exists 
$\rho,\alpha>0$  and $e\in E$ such that,
for any $\lambda>0$, 
\begin{equation*}
        \|e\|_E > \rho \quad \text{and} \quad \inf_{\|u\|_E = \rho} I_{\lambda}(u) \geq \alpha> 0=I_{\lambda}(0) > I_{\lambda}(e).
    \end{equation*}
\end{lemma}
\begin{proof}
	For any $u \in E$ with $\|u\|_E < 1$, by \eqref{eq:estimativaE} and the continuous embedding \eqref{Principal-Embedding}, we have    
    $$
        I_{\lambda}(u) 
        \geq
         \frac{1}{p} \|u\|_{E}^{p} - \frac{C_r}{r}\|u^+\|_E^{r}-\lambda \frac{C_s}{s}\|u^+\|_{E}^{s} \geq \frac{1}{p} \|u\|_{E}^{p} - \frac{C_r}{r}\|u\|_E^{r}-\lambda \frac{C_s}{s}\|u\|_{E}^{s} 
    $$
and being $p<\min\{r,s\}$, the existence of $\rho,\alpha>0$ is guaranteed.
Moreover, fixed $\varphi \in E\setminus\{0\}$ such that $\varphi_{+} \neq 0$ we get, as $t\to+\infty$, 
    \begin{align*}
        I_{\lambda}(t\varphi) &= \frac{t^p}{p}\int_{\mathbb{R}^N}|\nabla \varphi |^p+\frac{t^q}{q}\int_{\mathbb{R}^N}|\nabla \varphi |^q-\frac{t^r}{r}\int_{\mathbb{R}^N}\varphi_{+} ^r-\frac{\lambda t^s}{s}\int_{\mathbb{R}^N}\varphi_{+} ^s \to -\infty 
    \end{align*}
and since $\max\{r,s\}>p>q$, 
the proof is concluded.
\end{proof}

Actually the Mountain Pass Geometry does not depend on $\lambda>0$ and in view of Lemma~\ref{MP}, is well defined the minimax level
$$
0<c_\lambda:=\inf_{\gamma \in \Gamma}\max_{t\in[0,1]} I_\lambda(\gamma(t))
$$
where $\Gamma := \{\gamma \in C([0,1],E) : \gamma(0)=0 \text{ and } \gamma(1) = e\}$ 
and there exists a sequence $\{u_n\}\subset E$ such that 
\begin{equation}\label{PS}
I_\lambda(u_n)=c_\lambda+o_n(1)\quad\mbox{and}\quad I_{\lambda}'(u_n)=o_n(1).  
\end{equation}

\begin{lemma} \label{ltd1}

Assume that $q^* \leq r,s\leq p^*$ and $p<\min\{r,s\}$. Then, the sequence $\{u_n\}$ is bounded in $E$.

\end{lemma}

\begin{proof}
Set $d:= \min\{r,s\}$ and 
using \eqref{eq:usefull} we infer, 
\begin{equation*}
\begin{aligned}
c_{\lambda} + o_n(1) + o_n(1)\|u_n\|_E &\geq I_\lambda(u_n)-\frac{1}{d}I'_\lambda(u_n)[u_n] \\
&=\frac{d-p}{dp}\int_{\mathbb{R}^N}|\nabla u_n|^p+\frac{d-q}{dq}\int_{\mathbb{R}^N}|\nabla u_n|^q+ \frac{r-d}{rd}\int_{\mathbb{R}^N}u_{n+}^r \\
& \ \ + \lambda\frac{s-d}{ds}\int_{\mathbb{R}^N}u_{n+}^s\\
&\geq \frac{d-p}{dp}\int_{\mathbb{R}^N}|\nabla u_n|^p+\frac{d-q}{dq}\int_{\mathbb{R}^N}|\nabla u_n|^q\\
&\geq\frac{d-p}{dp} \left(\|\nabla u_n\|_{p}^p + \|\nabla u_n\|_{q}^q \right)\\
&\ge \frac{d-p}{dp} (\kappa \|u_n\|_E - 2).
\end{aligned}
\end{equation*}
Therefore 
$$
\left(\kappa \frac{d-p}{dp}-o_n(1)\right) \|u_n\|_E\leq c_\lambda+o_n(1)+2\frac{d-p}{dp},
$$
from which we concluded that $\{u_n\}$ is bounded in $E$.
\end{proof}

\begin{lemma} \label{minoracao}
Assume $1<q<p<N$ and $q^*\leq r,s\leq p^*$. If $\lambda>0$
and $\{u_n\}$  is a sequence satisfying \eqref{PS}, then there exists a constant $C_{\lambda}>0$ such that for $n$ large 
$$
\int_{\mathbb{R}^N}|u_n|^rdx+\int_{\mathbb{R}^N}|u_n|^sdx\geq C_{\lambda}>0.
$$
\end{lemma}
\begin{proof}It follows from Lemma \ref{ltd1} together with \eqref{PS} that
$$
\begin{aligned}
c_\lambda+o_n(1)&=I_\lambda(u_n)-\frac{1}{q}I'_\lambda(u_n)[u_n] \\
&=\frac{q-p}{pq}\int_{\mathbb{R}^N}|\nabla u_n|^p-\frac{q-r}{qr}\int_{\mathbb{R}^N}u_{n+}^r-\lambda\frac{q-s}{qs}\int_{\mathbb{R}^N}u_{n+}^s\\
&\leq\frac{r-q}{qr}\int_{\mathbb{R}^N}u_{n+}^r+\lambda \frac{s-q}{qs}\int_{\mathbb{R}^N}u_{n+}^s \\
&\leq \max\left\{\frac{r-q}{qr}, \lambda \frac{s-q}{qs} \right\}\left(\int_{\mathbb{R}^N}|u_n|^r + \int_{\mathbb{R}^N}|u_n|^s \right),
\end{aligned}
$$
which implies the desired result.
\end{proof}

\begin{proof}[Proof of Theorem~\ref{T1} \eqref{pmenor}]
By Lemma \ref{minoracao}, we know that  it does not happen $u_n \to 0$ in $L^r(\mathbb{R}^N)$ and $u_n \to 0$ in $L^s(\mathbb{R}^N)$. 

As a consequence, by the Lions Lemma \ref{Lions-Lemma} there exists a constant $C_t>0$ such that 
$$
\sup_{y\in \mathbb{R}^N}\int_{B_1(y)}|u_n|^t > C_t>0.
$$
Thus, we can find a sequence of points $\{y_n\}\subset \mathbb{R}^N$ such that 

$$
\int_{B_1(y_n)}|u_n|^t\geq C_t>0.
$$

Setting $v_n(x) = u_n(x+y_n)$, it follows that
\begin{equation}\label{Transladada}
\int_{B_1(0)}|v_n|^tdx\geq C_t>0.
\end{equation}
Using the translation invariance, it is easy to show that 
\begin{equation}\label{PS-Transladada}
  I_\lambda(v_n) \to c_{\lambda}
\quad \text{and} \quad 
I'_\lambda(v_n) \to 0.
\end{equation}
Applying Lemma \ref{ltd1}, since $p<\min\{r,s\}$,   we deduce that $\{v_n\}$ is bounded in $E$. 
Since $E$ is reflexive, there exists a subsequence (still denoted by $\{v_n\}$) such that 
\[
v_n \rightharpoonup v \quad \text{in } E.
\]
Consequently, 
\[
v_n \to v \quad \text{in } L^t_{\mathrm{loc}}(\mathbb{R}^{N}),
\]
for all $t \in [1,p^*)$, and therefore, it follows from \eqref{Transladada} that $v$ is nontrivial. In particular, from 
$$|v_{n+}(x)-v_{+}(x)| \leq |v_n(x)-v(x)| $$
for all $x \in \mathbb{R}^N$, we have 
$$v_{n+} \to v_{+} \quad \text{in} \quad L_{\mathrm{loc}}^{t}(\mathbb{R}^N).$$

We will prove that $v$ is a weak solution of \eqref{P} by testing the equation on test functions, as in \cite{F11}. In fact, for any $\varphi\in C_0^\infty (\mathbb{R}^N)$, it follows from \eqref{PS-Transladada} that  
$$
\int_{\mathbb{R}^N}\left(|\nabla v_n|^{p-2} \nabla v_n+|\nabla v_n|^{q-2}\nabla v_n\right)\nabla\varphi-\int_{\mathbb{R}^N}\left(v_{n+}^{r-1}+\lambda v_{n+}^{s-1}\right)\varphi=o_n(1).
$$
Since the embedding $E\hookrightarrow L^t_{\textrm{loc}}(\mathbb{R}^N)$ is compact, for $1\leq t<p^*$, and $v_n \rightharpoonup v$ in $E$, we readily have 
$$
\left|\int_{\mathbb{R}^N} \left(v_{n+}^{r-1} -v_{+}^{r-1}\right)\varphi \right| = o_n(1) \quad \text{and} \quad \left|\int_{\mathbb{R}^N} \left(v_{n+}^{s-1} -v_{+}^{s-1}\right)\varphi \right| = o_n(1).
$$
In fact, setting $K := \operatorname{supp} \varphi$ we infer
\begin{align*}
\left|\int_{\mathbb{R}^N} \left(v_{n+}^{r-1} -v_{+}^{r-1}\right)\varphi \right| \leq \int_{K} \left|v_{n+}^{r-1} -v_{+}^{r-1}\right|\left|\varphi\right| 
\leq \|\varphi\|_{\infty}  \int_{K} \left|v_{n+}^{r-1} -v_{+}^{r-1}\right|.
\end{align*}

Since  $v_{n_k+} \rightharpoonup v_{+}$ in $E$ and the embedding $E\hookrightarrow L^r_{\textrm{loc}}(\mathbb{R}^N)$ is compact, we see that $v_{n_k+} \to v_{+}$ in $L^r(K)$. Thus, there is $h \in L^{r}(K)$ such that 
$$
v_{n_k+} \to v_{+} \quad\text{a.e in $K$}
$$
and 
$$
|v_{n_k+}| \leq h \quad\text{a.e in $K$}
$$
up to subsequence. Consequently, again up to subsequence, it follows that 
$$\left|v_{n_k+}^{r-1} -
v_{+}^{r-1}\right| \to 0\quad\text{a.e in $K$}
$$
and 
$$\left|v_{n_k+}^{r-1} - v_{+}^{r-1}\right| \leq h^{r-1} + |v|^{r-1} \in L^{1}(K).
$$
By the Dominated Convergence of Lebesgue Theorem, we have  
$$
\int_{K} \left|v_{n+}^{r-1} -v_{+}^{r-1}\right| \to 0.
$$

Similarly we have 
$$
\int_{K} \left|v_{n+}^{s-1} -v_{+}^{s-1}\right| \to 0.
$$
Consequently,
$$
\int_{\mathbb{R}^N}\left(v_{n+}^{r-1}+\lambda v_{n+}^{s-1}\right)\varphi \to \int_{\mathbb{R}^N}\left(v_{+}^{r-1}+\lambda v_{+}^{s-1}\right)\varphi.
$$

Let us  prove that
\begin{equation}\label{Passandoolimite}
\int_{\mathbb{R}^N}\left(|\nabla v_n|^{p-2} \nabla v_n+|\nabla v_n|^{q-2}\nabla v_n\right)\nabla\varphi\rightarrow\int_{\mathbb{R}^N}\left(|\nabla v|^{p-2} \nabla v+|\nabla v|^{q-2}\nabla v\right)\nabla\varphi.
\end{equation}
In fact, it is enough to show that
$$
\int_{\mathbb{R}^N}|\nabla v_n|^{p-2}\nabla v_n \nabla\varphi \to \int_{\mathbb{R}^N}|\nabla v|^{p-2}\nabla v \nabla \varphi
$$
since the convergence 
$$
\int_{\mathbb{R}^N}|\nabla v_n|^{q-2}\nabla v_n \nabla\varphi \to \int_{\mathbb{R}^N}|\nabla v|^{q-2}\nabla v \nabla\varphi
$$
follows the same steps. 

We have
$$
\left|\int_{\mathbb{R}^N}|\nabla v_n|^{p-2}\nabla v_n \nabla\varphi - \int_{\mathbb{R}^N}|\nabla v|^{p-2}\nabla v \nabla \varphi \right| \leq \|\nabla \varphi\|_{\infty} \int_{K} \left||\nabla v_n|^{p-2}\nabla v_n -|\nabla v|^{p-2}\nabla v \right|
$$
so, our aim is to prove that the integral over $K := \operatorname{supp} \varphi$ tends to zero.

\smallskip

\textbf{Step 1:}
$\nabla v_n \to \nabla v$ in $L^p(B_1(0);\mathbb{R}^N)$.

Indeed, take $\psi \in C_0^{\infty} (\mathbb{R}^N)$ such that $\psi \equiv 1$ in $B_1(0)$ and $\psi \equiv 0$ in $\mathbb{R}^N \setminus B_2(0)$. 

Note that for $x,y \in \mathbb{R}^N$ if $\langle \cdot , \cdot \rangle$ denotes the standard scalar product in $\mathbb{R}^N$, it is true that
\[
\left\langle |x|^{p-2}x - |y|^{p-2}y,\; x-y \right\rangle
\ge
\begin{cases}
c_p |x-y|^p, & \text{if } p \ge 2, \\[6pt]
c_p \dfrac{|x-y|^2}{(|x|+|y|)^{2-p}}, & \text{if } 1<p<2,
\end{cases}
\]
where $c_p > 0$. Furthermore,
\begin{align*}
    0 &\leq \int_{B_1(0)} \left(|\nabla v_n|^{p-2}\nabla v_n - |\nabla v|^{p-2}\nabla v + |\nabla v_n|^{q-2}\nabla v_n - |\nabla v|^{q-2}\nabla v \right)(\nabla v_n - \nabla v) \\
    &\leq \int_{\mathbb{R}^N} \left[ \left(|\nabla v_n|^{p-2} \nabla v_n + |\nabla v_n|^{q-2} \nabla v_n \right) - \left(|\nabla v|^{p-2} \nabla v + |\nabla v|^{q-2} \nabla v \right) \right](\nabla v_n - \nabla v)\psi
\end{align*}
and 
\begin{align*}
    o_n(1) &= I'_{\lambda}(v_n)[v_n \psi] - I'_{\lambda}(v_n)[v \psi] \\
    &= \int_{\mathbb{R}^N} \left(|\nabla v_n|^{p-2} \nabla v_n + |\nabla v_n|^{q-2} \nabla v_n \right) (\nabla (v_n \psi) - \nabla (v \psi)) \\
    &+ \int_{\mathbb{R}^N} \left(v_{n+}^{r-1} + \lambda v_{n+}^{s-1} \right)(v \psi - v_n \psi) \\
    &= \int_{\mathbb{R}^N} \left(|\nabla v_n|^{p-2} \nabla v_n + |\nabla v_n|^{q-2} \nabla v_n \right) (\psi \nabla v_n + v_n \nabla \psi - \psi \nabla v - v \nabla \psi) + o_n(1) \\
    &= \int_{\mathbb{R}^N} \left(|\nabla v_n|^{p-2} \nabla v_n + |\nabla v_n|^{q-2} \nabla v_n \right)(\nabla v_n - \nabla v) \psi + o_n(1).
\end{align*}
Observe also that 
$$
\int_{B_1(0)} \left(|\nabla v_n|^{q-2}\nabla v_n - |\nabla v|^{q-2}\nabla v\right)(\nabla v_n - \nabla v) \geq 0.
$$
Moreover, since
$$
\left(|\nabla v|^{p-2} \nabla v + |\nabla v|^{q-2} \nabla v \right)\psi \in L^{p/(p-1)}(\mathbb{R}^N; \mathbb{R}^N)
$$
and 
$$
\nabla v_n \rightharpoonup \nabla v \quad \text{in} \quad L^p(\mathbb{R}^N; \mathbb{R}^N),$$
(recall that  $v_n \rightharpoonup v$ in $E$), we infer

\begin{equation*} \label{convergencia}
   \int_{\mathbb{R}^N}  \left(|\nabla v|^{p-2} \nabla v + |\nabla v|^{q-2} \nabla v \right) (\nabla v_n - \nabla v) \psi = o_n(1). 
\end{equation*}

Thus, if $p\geq 2$, we have
\begin{align*}
o_n(1) &= I'_{\lambda}(v_n)[v_n \psi] - I'_{\lambda}(v_n)[v \psi]\\
&= \int_{\mathbb{R}^N} \left(|\nabla v_n|^{p-2} \nabla v_n + |\nabla v_n|^{q-2} \nabla v_n \right)(\nabla v_n - \nabla v) \psi + o_n(1) \\
&= \int_{\mathbb{R}^N} \left(|\nabla v_n|^{p-2} \nabla v_n + |\nabla v_n|^{q-2} \nabla v_n - \left(|\nabla v|^{p-2} \nabla v + |\nabla v|^{q-2} \nabla v \right) \right)(\nabla v_n - \nabla v) \psi + o_n(1) \\
&+ \int_{\mathbb{R}^N}  \left(|\nabla v|^{p-2} \nabla v + |\nabla v|^{q-2} \nabla v \right) (\nabla v_n - \nabla v) \psi \\
&\geq \int_{B_1(0)} \left(|\nabla v_n|^{p-2}\nabla v_n - |\nabla v|^{p-2}\nabla v + |\nabla v_n|^{q-2}\nabla v_n - |\nabla v|^{q-2}\nabla v \right)(\nabla v_n - \nabla v) + o_n(1) \\
&\geq \int_{B_1(0)} \left(|\nabla v_n|^{p-2}\nabla v_n - |\nabla v|^{p-2}\nabla v \right)(\nabla v_n - \nabla v) \psi + o_n(1) \\
&\geq c_p \int _{B_1(0)} \left|\nabla v_n - \nabla v \right|^p + o_n(1)
\end{align*}
and the result is immediate.

If $1<p<2$, then 
$$
o_n(1) \geq c_p \int_{B_1(0)} \frac{|\nabla v_n - \nabla v|^2}{\left(|\nabla v_n| + |\nabla v| \right)^{2-p}} + o_n(1).
$$

Consequently,
\begin{align*}
 \int_{B_1(0)} |\nabla v_n - \nabla v|^p &= \int_{B_1(0)} \frac{|\nabla v_n - \nabla v|^p}{\left(|\nabla v_n| +|\nabla v|\right)^{(2-p)p/2}} \cdot \left(|\nabla v_n| +|\nabla v|\right)^{(2-p)p/2} \\
 &\leq \left(\int_{B_1(0)} \frac{|\nabla v_n - \nabla v|^2}{\left(|\nabla v_n| +|\nabla v|\right)^{2-p}} \right)^{p/2} \left(\int_{B_1(0)} \left(|\nabla v_n| +|\nabla v|\right)^{p}\right)^{(2-p)/2} \\
 &\leq C o_n(1).
\end{align*}
Again, the result follows.

\smallskip

\textbf{Step 2}: $\displaystyle\int_{K} \left||\nabla v_n|^{p-2}\nabla v_n -|\nabla v|^{p-2}\nabla v \right| \to 0$.

\smallskip
Take a sufficiently large $R>0$ such that $K \subset B_R(0)$. Based on the construction done in the previous step, we have that $\nabla v_n \to \nabla v$ in $L^p(B_R(0); \mathbb{R}^N)$ and, therefore, 
$$
\int_{K} \left|\nabla v_n - \nabla v \right|^p \leq \int_{B_R(0)} \left|\nabla v_n - \nabla v \right|^p \to 0,
$$ 
that is $\nabla v_n \to \nabla v$ in $L^p(K;\mathbb{R}^N)$.

Therefore we easily conclude that
$$
\int_{\mathbb{R}^N}|\nabla v_n|^{p-2}\nabla v_n \nabla\varphi \to \int_{\mathbb{R}^N}|\nabla v|^{p-2}\nabla v \nabla \varphi.
$$

So, \eqref{Passandoolimite} is valid and, consequently,
$$
\int_{\mathbb{R}^N}\left(|\nabla v|^{p-2} v+|\nabla v|^{q-2}\nabla v\right)\nabla\varphi-\int_{\mathbb{R}^N}\left(v_{+}^{r-1} +\lambda v_{+}^{s-1} \right)\varphi=0, \quad \forall \varphi\in C_0^\infty(\mathbb{R}^N).
$$
Due to the density result of Proposition \ref{Densidade}, we concluded that 
\begin{equation} \label{identidade}
\int_{\mathbb{R}^N}\left(|\nabla v|^{p-2} v+|\nabla v|^{q-2}\nabla v\right)\nabla\varphi-\int_{\mathbb{R}^N}\left(v_{+}^{r-1} +\lambda v_{+}^{s-1} \right)\varphi=0, \quad \forall \varphi\in E,
\end{equation}
that is, $v$ is a nontrivial weak solution of \eqref{P}. Taking $\varphi = v_{-}$ in \eqref{identidade} we obtain $v_{-} = 0$ in $E$ and so, $v \geq 0$.
\end{proof}

\subsection{Proof of  Theorem \ref{T1} (\ref{pnomeio})}

Take $\Psi \in C_{0}^{\infty}(\mathbb{R})$ such that $0 \leq \Psi \leq 1$, $\Psi \equiv 1$ in $[0,1]$ and $\Psi \equiv 0$ in $[2,+\infty)$. Define, for any $L>0$, 
\begin{equation*}
    H_{L}(u) = \frac{1}{s} \Psi_{L}\left(u\right) \int_{\mathbb{R}^N} u_{+}^{s}, \quad u \in E,
\end{equation*}
where 
\begin{equation*}
    \Psi_{L}(u) := \Psi \left(\frac{\|u\|_E^p}{L^p} \right), \quad u \in E.
\end{equation*}
Clearly, $H_L$ is well defined,
 supported in $\{u\in E: \| u\| \le 2L\}$
and belongs to $C^1(E;\mathbb{R})$ with 
\begin{align*}
    H'_{L}(u)[\varphi] &= \Psi_L(u) \int_{\mathbb{R}^N}u_{+}^{s-1}\varphi +\\
    &+\frac{p}{sL^p} \Psi'\left(\frac{\|u\|_E^{p}}{L^p}\right)\left(\int_{\mathbb{R}^N} |\nabla u|^{p-2} \nabla u \nabla \varphi + \left(\int_{\mathbb{R}^N}|\nabla u|^{q}\right)^{\frac{p}{q}-1}\int_{\mathbb{R}^N} |\nabla u|^{q-2} \nabla u \nabla \varphi\right) \int_{\mathbb{R}^N} u^s_+.
\end{align*}
Define the truncated functional $J_{\lambda, L} : E \to \mathbb{R}$ by 
\begin{equation*}
    J_{\lambda, L}(u) = \frac{1}{p} \int_{\mathbb{R}^N} |\nabla u|^p + \frac{1}{q} \int_{\mathbb{R}^N} |\nabla u|^q - \frac{1}{r}\int_{\mathbb{R}^N} u_{+}^r - \lambda H_L(u),
\end{equation*}
which is in $ C^{1}(E;\mathbb{R})$, with 
\begin{equation*}
    J'_{\lambda, L}(u)[\varphi] = \int_{\mathbb{R}^N} \left(|\nabla u|^{p-2} \nabla u + |\nabla u|^{q-2} \nabla u\right) \nabla \varphi - \int_{\mathbb{R}^N} u_{+}^{r-1} \varphi -\lambda H'_{L}(u)[\varphi],
\end{equation*}
and a critical point $u \in E$ of $J_{\lambda, L}$ is a critical point of $I_{\lambda}$ when $\|u\|_E \leq L$.

For the next result we will use as the comparison functional
    \begin{equation*}
        I_{0}(u) = \frac{1}{p}\int_{\mathbb{R}^N}|\nabla u|^p+\frac{1}{q}\int_{\mathbb{R}^N}|\nabla u|^q-\frac{1}{r}\int_{\mathbb{R}^N}u_{+}^r, \quad u\in E.
    \end{equation*}
    Of course $I_0(u)\ge J_{\lambda, L}(u)$.

\begin{lemma} \label{GPM-positivo}
    If $s<p<r$, there exist  $\lambda_1 > 0$, $\alpha, \rho_0 > 0$ and $e \in E$ 
    such that, for every  $L>0$ and  $\lambda \in (0,\lambda_1)$ it holds
    \begin{equation*}
        \|e\|_E > \rho_0 \quad \text{and} \quad \inf_{\|u\|_E = \rho_0} J_{\lambda,L}(u) \geq \alpha> 0=J_{\lambda,L}(0) > J_{\lambda,L}(e),
    \end{equation*}
    and
    $$
      \inf_{\|u\|_E = \rho_0} I_{\lambda}(u) \geq \alpha> 0=I_{\lambda}(0) > I_{\lambda}(e).
     $$

\end{lemma}
\begin{proof}
    Indeed, if $\|u\|_E = \rho < 1$, by \eqref{eq:estimativaE} we get
    \begin{equation*}
        I_0(u)\ge J_{\lambda, L}(u) \geq I_{\lambda}(u)
        \geq \frac{1}{p}\|u\|_E^{p} - \frac{C_r}{r}\|u\|_E^{r} - \lambda \frac{C_s}{s}\|u\|_{E}^s = \rho^p \left(\frac{1}{p} - \frac{C_r}{r}\rho^{r-p} - \lambda \frac{C_s}{s}\rho^{s-p}\right).
    \end{equation*}
  
    Choose now $\rho_0\in(0,1)$  such that 
    $$\frac{1}{p}-\frac{C_r}{r}\rho_0^{r-p}>\frac{1}{2p}$$
    and take 
    \begin{equation*}
        \lambda_1 < \frac{s\rho_{0}^{p-s}}{2pC_s}.
    \end{equation*}
    Then, for $\lambda \in (0,\lambda_1)$ and $\|u\|_E = \rho_0$,
    \begin{equation*}
         J_{\lambda, L}(u) \geq I_{\lambda}(u)
        \geq \rho_{0}^p \left(\frac{1}{2p} - \lambda\frac{C_s}{s}\rho_{0}^{s-p}\right) \geq  \rho_{0}^p \left(\frac{1}{2p} - \lambda_1\frac{C_s}{s}\rho_{0}^{s-p}\right) =:\alpha> 0.
    \end{equation*}
    Furthermore, fixed $\varphi \in E\setminus\{0\}$ such that $\varphi\ge 0$ we have
    \begin{equation*}
        I_{0}(t\varphi) = \frac{t^p}{p} \int_{\mathbb{R}^N} |\nabla \varphi|^p + \frac{t^q}{q} \int_{\mathbb{R}^N} |\nabla \varphi|^q - \frac{t^r}{r} \int_{\mathbb{R}^N} \varphi^r \to -\infty
    \end{equation*}
    as $t \to + \infty$, and so we easily get $e\in E$ with the desired properties.

\end{proof}

It follows from  Lemma \ref{GPM-positivo} that for any $\lambda \in (0,\lambda_1)$ and $L>0$ there exists a sequence $\{u_n\} := \{u_n^{\lambda, L}\}\subset E$ such that
\begin{equation} \label{J-sequence1}
    J_{\lambda,L}(u_n) = c_{\lambda,L} + o_n(1) \quad \text{and} \quad J'_{\lambda,L}(u_n) = o_n(1)
\end{equation}
where as usual
$$ c_{\lambda, L} := \inf_{\gamma \in \Gamma} \max_{t \in [0,1]} J_{\lambda, L}(\gamma(t))\ge \alpha>0$$
and $\Gamma := \left\{\gamma \in C\left([0,1]; E \right): \gamma(0) = 0,  \gamma(1) = e \right\}$.
Note that the above proof says that the class $\Gamma$
is the same for $I_0, I_\lambda$ and $J_{\lambda,L}$ and so we have
\begin{equation} \label{minimaxs1}
    0 < c_{\lambda} := \inf_{\gamma \in \Gamma} \max_{t \in [0,1]} I_{\lambda}(\gamma(t)) \leq c_{\lambda, L}  \leq c_{0} = \inf_{\gamma \in \Gamma} \max_{t \in [0,1]} I_{0}(\gamma(t)) < +\infty,
\end{equation}
and moreover $c_\lambda\to c_0$ as $\lambda\to 0.$

\begin{lemma} \label{ltd-truncado1}
    There exists $L_0>0$  and $\lambda_0 = \lambda_0(L_0) \in (0,\lambda_1)$ such that,
    for any $L\ge L_0$ and any $\lambda\in(0,\lambda_0)$
      the sequence in \eqref{J-sequence1}
      for the functional $J_{\lambda,L}$ is bounded by $L$.

\end{lemma}

\begin{proof}
We argue by contraddition.
Suppose  that for any $L>0$
and $\lambda \in (0,\lambda_1)$,
the sequence $\{u_n\}$ associated to $J_{\lambda,L}$ satisfying \eqref{J-sequence1} is such that 
$\|u_n\|_E > L$.
In particular this holds for
the pair $(L,\lambda)$ where
    $L$ is so large that
\begin{equation}\label{eq:L1}
    L> \frac{pr}{\kappa (r-p)}(c_{0} + 2\frac{r-p}{pr} + 1)
\end{equation}
 and in such a way that $\lambda$ given by
\begin{equation}\label{eq:lambda1}
    \lambda = \left(\frac{2^pp\|\Psi'\|_{\infty} + r - s}{sr}  C_s^s (2L)^s\right)^{-1}
\end{equation}
 belongs to $(0,\lambda_1)$.

Here, $c_0 > 0$ is the constant in \eqref{minimaxs1}, $\kappa>0$ is the constant which appear in \eqref{eq:usefull} and $C_s$ is the embedding constant of $E$ in $L^s(\mathbb{R}^N)$.

If it were
$\|u_n\|_E \geq 2L$ for all $n \in \mathbb{N}$ sufficiently large,
then, since the truncated term is zero,
\begin{align*}
    c_{0} + o_n(1) + o_n(1)\|u_n\|_E &\geq c_{\lambda, L} + o_n(1) + o_n(1)\|u_n\|_E \\
    &= J_{\lambda, L}(u_n) - \frac{1}{r} J'_{\lambda, L}(u_n)[u_n] \\
    &= \frac{r-p}{pr}\int_{\mathbb{R}^N}|\nabla u|^p +  \frac{r-q}{rq} \int_{\mathbb{R}^N}|\nabla u|^q \\
    &\geq \frac{r-p}{pr} \left(\kappa\|u_n\|_E-2 \right). 
\end{align*}
Consequently, for all $n \in \mathbb{N}$ large, 
\begin{eqnarray*}
   c_{0} + 2\frac{r-p}{pr} + 1
   &\ge& c_{\lambda,L} + 2\frac{r-p}{pr} + o_n(1) \\ &\geq& \left(\kappa\frac{r-p}{pr} - o_n(1)\right) \|u_n\|_E \\
    &\geq &\left(\kappa\frac{r-p}{pr} - o_n(1)\right) 2L   
\end{eqnarray*}
and so taking the limit as $n\to +\infty$
we get a contradicts with \eqref{eq:L1}.

So it has to be  $L < \|u_n\|_E < 2L$. But then again, 
\begin{align*}
    c_{0} + o_n(1) + o_n(1)\|u_n\|_E &= c_{\lambda, L} + o_n(1) + o_n(1)\|u_n\|_E \\
    &\geq J_{\lambda, L}(u_n) - \frac{1}{r} J'_{\lambda, L}(u_n)[u_n] \\
    &= \frac{r-p}{pr} \int_{\mathbb{R}^N}|\nabla u|^p +  \frac{r-q}{qr} \int_{\mathbb{R}^N}|\nabla u|^q + \lambda \left(\frac{1}{r}H'_L(u_n)u_n - H_L(u_n) \right) \\
    &\ge \frac{r-p}{pr}\left(\kappa\|u_n\|_E-2 \right) +\\
    & \ \ + \lambda \left(\frac{p}{srL^p} \Psi'\left(\frac{\|u_n\|_E^p}{L^p}\right)\|u_n\|_{E}^{p} + \frac{s-r}{rs}\Psi_L(u_n) \right) \int_{\mathbb{R}^N}u_{n+}^s \\
    &\geq  \frac{r-p}{pr}\left(\kappa\|u_n\|_E-2 \right) + \lambda \left(-\frac{p}{srL^p} \|\Psi'\|_{\infty} \|u_n\|_E^p + \frac{s-r}{sr}\right)C_s^s \|u_n\|_E^s \\
    &=\frac{r-p}{pr}\left(\kappa\|u_n\|_E-2 \right) - \lambda \left(\frac{p}{srL^p} \|\Psi'\|_{\infty} \|u_n\|_E^p + \frac{r-s}{sr}\right)C_s^s \|u_n\|_E^s.
\end{align*}
Thus, for all $n \in \mathbb{N}$ sufficiently large, we obtain 
\begin{align*}
    c_{0} + 2\frac{r-p}{pr} + o_n(1) &\geq \left( \kappa\frac{r-p}{pr} - o_n(1) \right) \|u_n\|_E - \lambda \left(\frac{p \|\Psi'\|_{\infty}}{srL^p} (2L)^p + \frac{r-s}{sr}\right) C_s^s \|u_n\|_E^{s} \\
    &\geq \left( \kappa\frac{r-p}{pr} - o_n(1) \right) L - \lambda \frac{2^pp \|\Psi'\|_{\infty} + r-s}{sr}  C_s^s (2L)^s.
\end{align*}
Consequently, passing to the limit, by \eqref{eq:lambda1}
we get
\begin{equation*}
    c_{0} + 2\frac{r-p}{pr} \geq \kappa\frac{r-p}{pr} L - \lambda \frac{2^pp \|\Psi'\|_{\infty} +r-s }{sr}  C_s^s (2L)^s=
    \kappa\frac{r-p}{pr} L-1
\end{equation*}
that again contradicts \eqref{eq:L1}. Then the Lemma is proved.

\end{proof}

In particular, the proof of Lemma \ref{ltd-truncado1} shows that any sequence $\{w_n\} \subset E$ satisfying \eqref{J-sequence1} also fulfills the assumptions required to apply Lemma \ref{ltd-truncado1}, and therefore it is conclusion remains valid for $\{w_n\}$. From this point on, we fix $L>0$ sufficiently large and implicitly assume that it is chosen so that Lemma \ref{ltd-truncado1} is always applicable.

The next result shows that the sequence $\{u_n\}$ does not vanish in $L^r(\mathbb R^N)\cap L^s(\mathbb R^N)$.
\begin{lemma}
    For any $\lambda \in (0,\lambda_0)$ and $L\ge L_0$ fixed, there exists $C_{\lambda} > 0$ such that
    \begin{equation*}
        \int_{\mathbb{R}^N} |u_n|^r + \int_{\mathbb{R}^N} |u_n|^s \geq C_\lambda.
    \end{equation*}
\end{lemma}

\begin{proof}
    For $n \in \mathbb{N}$ sufficiently large, by Lemma \ref{ltd-truncado1}, we have $J_{\lambda, L}(u_n) = I_{\lambda}(u_n)$. Thus, 
    \begin{align*}
        c_\lambda + o_n(1) &\leq c_{\lambda, L} + o_n(1) = J_{\lambda, L}(u_n) -\frac{1}{q} J'_{\lambda, L}(u_n)u_n \\
        &=\frac{q-p}{pq}\int_{\mathbb{R}^N}|\nabla u_n|^p
        +\frac{r-q}{qr}\int_{\mathbb{R}^N}u_{n+}^r
        +\lambda \frac{s-q}{sq}\int_{\mathbb{R}^N}u_{n+}^s \\
        &\leq \frac{r-q}{qr}\int_{\mathbb{R}^N}|u_n|^r 
        + \lambda_0\frac{s-q}{sq}\int_{\mathbb{R}^N}|u_n|^s,
    \end{align*}
    and the result follows.
\end{proof}
Thus, the Lions Lemma tells us that
\begin{equation*}
   \sup_{y\in \mathbb{R}^N}\int_{B_1(y)}|u_n|^t > C_t>0
\end{equation*}
and there exists a sequence $\{y_n\} \subset \mathbb{R}^N$ such that 
\begin{equation} \label{ltdinf}
    \int_{B_1(y_n)}|u_n|^t \geq C_t.
\end{equation}
Then the translated sequence $v_n(x):= u_n(x + y_n)$ has norm less then $L$. Consequently, there exists $v_{\lambda,L} = v \in E$ such that $v_n \rightharpoonup v$ in $E$ (up to subsequence). This together with \eqref{ltdinf} implies 
\begin{equation*}
    \int_{B_1(0)}|v|^t \geq C_t,
\end{equation*}
that implies $v \neq 0$ in $E$. On the other hand, since 
\begin{equation*}
    J_{\lambda, L}(v_n) = I_{\lambda}(v_n) = I_{\lambda}(u_n) = J_{\lambda, L}(u_n) \to c_{\lambda, L}
\end{equation*}
 and 
 \begin{equation*}
     \|J'_{\lambda, L}(v_n)\| = \|I'_{\lambda}(v_n)\|=\|I'_{\lambda}(u_n)\| = \|J'_{\lambda, L}(u_n)\| = o_n(1)
 \end{equation*}
 for all sufficiently large $n \in \mathbb{N}$, and following the same construction as before, we obtain that $v$ is a critical point of the energy functional $J_{\lambda, L}$. Moreover, since $v_n \rightharpoonup v$ in $E$ and $\|v_n\|_E \leq L$ for all sufficiently large $n$, it follows by weak lower semicontinuity that $\|v\|_E \leq L$. Consequently, $v$ lies within the region where the truncation is inactive, and hence it is also a critical point of $I_{\lambda}$. Therefore, $v$ is a weak solution of problem \eqref{P} and, clearly, $v$ is a non-negative function.

\subsection{Proof of Theorem \ref{th:supercritical}
}

Now, let us assume that $q^* < r < p^* < s$ and $p<r$. Clearly the functional $I_{\lambda}$ is not well defined. Thus, motivated by  paper \cite{AR24}, we prove the existence of a weak solution to problem \eqref{P} by means of a suitable truncation technique.

For any $K>1$, define
$$
\phi_{\lambda, K}(t) =
\begin{cases} 
0 & \text{if } t < 0, \\
t^{r-1} + \lambda t^{s-1} & \text{if } 0 \leq t \leq K, \\
(1 + \lambda K^{s-r}) t^{r-1} & \text{if } t > K.
\end{cases}
$$
Clearly, 
\begin{equation} \label{des1}
    0 \leq \phi_{\lambda, K}(t) \leq (1 + \lambda K^{s-r})|t|^{r-1}\quad \forall t \in \mathbb{R}.
\end{equation}

Consider the truncated problem
\begin{equation}\label{P-truncado}
    -\Delta_p u-\Delta_q u= \phi_{\lambda,K}(u)\quad\mbox{in} \quad \mathbb{R}^N.
\tag{$\mathcal{P}_T$}
\end{equation}

\begin{definition} 
    We say that a function $u \in E$ is a weak solution of equation \eqref{P-truncado} if 
\begin{equation} \label{weak-truncado}
    \int_{\mathbb{R}^N}\left(|\nabla u|^{p-2}\nabla u+|\nabla u|^{q-2}\nabla u\right)\nabla\varphi-\int_{\mathbb{R}^N} \phi_{\lambda, K}(u)\varphi=0,
\end{equation}
for all $\varphi \in E$.
\end{definition}

We observe that weak solution of \eqref{P-truncado} are precisely the critical points of the functional energy defined by 
$$
J_{\lambda, K}(u) = \frac{1}{p}\int_{\mathbb{R}^N}|\nabla u|^p+\frac{1}{q}\int_{\mathbb{R}^N}|\nabla u|^q-\int_{\mathbb{R}^N}\Phi_{\lambda, K}(u), \quad u \in E,
$$
where 
$$
\Phi_{\lambda, K}(t)=\int_{0}^{t} \phi_{\lambda, K}(\tau) d \tau.
$$
Clearly by \eqref{des1} and $E \hookrightarrow L^r(\mathbb{R}^N)$, $J_{\lambda, K}$ is well defined. Furthermore,
\begin{equation*}
    0 \leq \Phi_{\lambda, K}(t) \leq \left(\frac{1 + \lambda K^{s-r}}{r}\right)|t|^r,
\end{equation*}
for all $t \in \mathbb{R}$. It is also not difficult to verify that if $u \in E$ is a weak solution of \eqref{P-truncado} then necessarily $u \geq 0$. It suffices to take $\varphi = u^{-}$ in the identity \eqref{weak-truncado}.

\begin{lemma} \label{GPMsuper}
Assume that $q^* < r < p^* < s$ and $p<r$. For any $K > 1$, there are $M_K > 0$, $\rho > 0$, $\alpha > 0$ and $e \in E$ (that do not depend on $\lambda$ and $K$) such that 
$$
\|e\|_E> \rho \quad \text{and} \quad \inf_{\|u\|_E=\rho} J_{\lambda, K}(u) \geq \alpha > 0 = J_{\lambda, K}(0) >  J_{\lambda, K}(e),
$$
for all $\lambda \in (0,M_K)$.
\end{lemma}
\begin{proof}
    For $u \in E$ satisfying $\|u\|_E =\rho < 1$, we have
\begin{align*}
J_{\lambda, K}(u) &\geq \frac{1}{p} \|u\|_{E}^{p} - \int_{\mathbb{R}^N} \Phi_K(u) \geq \frac{1}{p} \|u\|_{E}^{p} - \left(\frac{1 + \lambda K^{s-r}}{r}\right)\int_{\mathbb{R}^N}|u|^r \\
&\geq \frac{1}{p} \|u\|_{E}^{p} - \left(\frac{1 + \lambda K^{s-r}}{r}\right)C\|u\|_{E}^r =\rho^{p}\left(\frac{1}{p} - \frac{1 + \lambda K^{s-r}}{r}C \rho^{r-p} \right).
\end{align*}
Take $M_K = 1/K^s > 0$. If $\lambda \in (0, M_K)$, then
$$
J_{\lambda, K}(u) \geq \rho^{p}\left(\frac{1}{p} - \frac{1 + K^{-r}}{r}C \rho^{r-p} \right)>\rho^{p}\left(\frac{1}{p} - \frac{2}{r}C \rho^{r-p} \right)
$$
Since $r-p > 0$ and $\rho^{r-p} \to 0$ as $\rho \to 0^+$, taking a sufficiently small $\rho>0$ such that
$$
\frac{1}{p}-\frac{2}{r}C\rho^{r-p} > 0,
$$
 we obtain 
$$
J_{\lambda, K}(u) > \rho^{p}\left(\frac{1}{p} - \frac{2}{r}C \rho^{r-p} \right) := \alpha > 0.
$$

Now, 
fixed $\varphi \in C_{0}^{\infty}(B_1(0))$ with $\varphi > 0$,
for all sufficiently large $t>0$  it holds true 
\begin{align*}
    J_{\lambda, K}(t\varphi) &= \frac{t^p}{p} \int_{\mathbb{R}^N}|\nabla\varphi|^p + \frac{t^q}{q} \int_{\mathbb{R}^N}|\nabla\varphi|^q - \int_{\mathbb{R}^N} \Phi_{\lambda, K}(t\varphi) \\
    &= \frac{t^p}{p} \int_{\mathbb{R}^N}|\nabla\varphi|^p + \frac{t^q}{q} \int_{\mathbb{R}^N}|\nabla\varphi|^q + \\
    &- \int_{B_1(0)} \left(  \int_{0}^{K} (\tau^{r-1} + \lambda\tau^{s-1})d\tau + \int_{K}^{t\varphi}(1+\lambda K^{s-r})\tau^{r-1} d\tau \right)  \\
    &=  \frac{t^p}{p} \int_{\mathbb{R}^N}|\nabla\varphi|^p + \frac{t^q}{q} \int_{\mathbb{R}^N}|\nabla\varphi|^q - \frac{(1+\lambda K^{s-r})t^r}{r}\int_{B_1(0)} \varphi^r + \lambda  \frac{s-r}{rs}K^s|B_1(0)|\\
    &\leq \frac{t^p}{p} \int_{\mathbb{R}^N}|\nabla\varphi|^p + \frac{t^q}{q} \int_{\mathbb{R}^N}|\nabla\varphi|^q - \frac{t^r}{r}\int_{B_1(0)} \varphi^r + \lambda\frac{s-r}{rs}K^s|B_1(0)|.
\end{align*}
Again, if $\lambda \in (0,M_K)$, then 
\begin{align*}
    J_{\lambda, K}(t \varphi) &\leq \frac{t^p}{p} \int_{\mathbb{R}^N}|\nabla\varphi|^p + \frac{t^q}{q} \int_{\mathbb{R}^N}|\nabla\varphi|^q - \frac{t^r}{r}\int_{B_1(0)} \varphi^r + \frac{s-r}{rs}|B_1(0)| \\
    &=t^r \left( \frac{t^{p-r}}{p} \int_{\mathbb{R}^N}|\nabla\varphi|^p + \frac{t^{q-r}}{q} \int_{\mathbb{R}^N}|\nabla\varphi|^q - \frac{1}{r}\int_{B_1(0)} \varphi^r + t^{-r} \frac{s-r}{rs}|B_1(0)|\right) \to -\infty
\end{align*}
as $t \to +\infty$, because $r>p>q$. Thus, there is a sufficiently large $t_0 > 0$ such that $\|t_0 \varphi\|_E > \rho$ and
$$
\frac{t_{0}^p}{p} \int_{\mathbb{R}^N}|\nabla\varphi|^p + \frac{t_{0}^q}{q} \int_{\mathbb{R}^N}|\nabla\varphi|^q - \frac{t_{0}^r}{r}\int_{B_1(0)} \varphi^r +  \frac{s-r}{rs}|B_1(0)| < 0.
$$
Taking $e = t_0 \varphi \in E$, it holds  $\|e\|_E > \rho$ and  $J_{\lambda, K}(e) < 0$, for all $\lambda \in (0,M_K)$.
This concludes the proof. 
\end{proof}

By  Lemma \ref{GPMsuper}, for any $\lambda \in (0, M_K)$ there exists a sequence of points $u_n = u_{\lambda, K}^n \in E$ such that 
$$
J_{\lambda, K}(u_n) = c_{\lambda, K} + o_n(1) \quad \text{and} \quad J_{\lambda, K}(u_n)=o_n(1),
$$
where 
$$
c_{\lambda,K}:=\inf_{\gamma \in \Gamma} \max_{t\in[0,1]} J_{\lambda,K}(\gamma(t))
$$
and $\Gamma := \{\gamma \in C([0,1],E) : \gamma(0)=0 \text{ and } \gamma(1) = e\}$.

\begin{lemma} \label{ltd}
Assume that $q^*<r<p^*<s$ and $p<r$. Then $\{u_n\}$ is bounded in $E$ by a constant that does not depend on $\lambda$ and $K$.
\end{lemma}
\begin{proof}
    Indeed, we observed that
    $$
    \Phi_{\lambda, K}(t) - \frac{1}{r} \phi_{\lambda, K}(t)t \leq 0, \quad \forall t\in \mathbb{R}.
    $$
    Consequently,
\begin{align*}
    c_{\lambda, K} + o_n(1) + o_n(1)\|u_n\|_{E} &\geq J_{\lambda, K}(u_n) - \frac{1}{r} J'_{\lambda, K}(u_n)[u_n] \\
    &= \frac{r-p}{rp}\int_{\mathbb{R}^N}|\nabla u_n|^p + \frac{r-q}{rq}\int_{\mathbb{R}^N}|\nabla u_n|^q +\\
    &- \int_{\mathbb{R}^N} \left(\Phi_{\lambda, K}(u_n) - \frac{1}{r}\phi_{\lambda, K}(u_n)u_n \right) \\
    &\geq \frac{r-p}{rp} \left( \int_{\mathbb{R}^N}|\nabla u_n|^p + \int_{\mathbb{R}^N}|\nabla u_n|^q \right) \\
    &\geq \kappa \frac{r-p}{rp}\|u_n\|_E - 2 \frac{r-p}{rp}.
\end{align*}
We deduce
\begin{equation}\label{limited}
    c_{\lambda, K} + o_n(1) +  2 \frac{r-p}{rp} \geq \left( \kappa \frac{r-p}{rp} - o_n(1) \right) \|u_n\|_E.
\end{equation}

On the other hand, since $\lambda_1 < \lambda_2$ implies $J_{\lambda_1, K} \geq J_{\lambda_2, K}$ and the family $\Gamma$ is the same for all $\lambda \in (0,M_K)$, we have $c_{\lambda_2, K} \leq c_{\lambda_1, K}$ and 
\begin{equation} \label{minimax}
    c_{\lambda, K} = \inf_{\gamma \in \Gamma} \max_{t \in [0,1]} J_{\lambda, K} (\gamma(t)) \leq \inf_{\gamma \in \Gamma} \max_{t \in [0,1]} J_{0}(\gamma(t)) < +\infty,
\end{equation}
where
$$
J_{0}(u) = \frac{1}{p}\int_{\mathbb{R}^N}|\nabla u|^p+\frac{1}{q}\int_{\mathbb{R}^N}|\nabla u|^q-\int_{\mathbb{R}^N}\left( \int_{0}^{t} \phi_0(\tau)d\tau \right), \quad u \in E,
$$
and $\phi_0 = \phi_{\lambda, K}$ with $\lambda = 0$, that is,
$$
\phi_{0}(t) =
\begin{cases} 
0 & \text{if } t < 0, \\
t^{r-1} & \text{if } t \geq 0.
\end{cases}
$$
Therefore, from \eqref{limited} and \eqref{minimax} the conclusion follows.
\end{proof}

\begin{lemma} \label{minorando}
    There exists a constant $C_{\lambda, K}>0$ such that for all sufficiently large $n \in \mathbb{N}$, we have
    $$
    \int_{\mathbb{R}^N} |u_n|^r \geq C_{\lambda,K}.
    $$
\end{lemma}
\begin{proof}
    Indeed, using the Lemma \ref{ltd}, we have 
    \begin{align*}
        c_{\lambda, K} - o_n(1) &= J_{\lambda, K}(u_n) - \frac{1}{q} J'_{\lambda, K}(u_n)[u_n] =\frac{q-p}{pq}  \int_{\mathbb{R}^N} |\nabla u_n|^{p} + \int_{\mathbb{R}^N} \left(\frac{1}{q} \phi_K(u_n)u_n-\Phi_K(u_n) \right) \\
        &\leq \int_{\mathbb{R}^N} \left(\frac{1}{q} |\phi_K(u_n)u_n| +|\Phi_K(u_n)|\right) \leq \int_{\mathbb{R}^N} \left(\frac{1+\lambda K^{s-r}}{q}|u_n|^r + \frac{1+\lambda K^{s-r}}{r}|u_n|^r\right) \\
        &\leq \left(1+\lambda K^{s-r}\right)\frac{r+q}{rq} \int_{\mathbb{R}^N}|u_n|^{r},
    \end{align*}
    concluding the proof.
\end{proof}

As a consequence of Lemma \ref{minorando}, the sequence $\{u_n\}$ does not converge to zero in $L^r(\mathbb{R}^N)$. The Lions Lemma ensures that
$$
\sup_{y \in \mathbb{R}^N} \int_{B_{1}(y)} |u_n|^t > C_1,
$$
for some constant $C_1=C_1(t, \lambda, K)>0$ and for any $t \in [q^*,p^*)$.
Thus, following the same ideas as the previous subsection, we obtain that the bounded sequence $v_n(x) =u_n(x+y_n)$, where $\{y_n\} \subset \mathbb{R}^N$ is such that 
$$
\int_{B_1(y_n)}|u_n|^t\geq C_1>0,
$$
converges weakly (up to subsequence) to a non-trivial function $v_{\lambda, K} \in E$, and with a similar calculation done previously, it can be concluded that is a weak solution of \eqref{P-truncado}.

\begin{lemma} There are sufficiently large $K_0 > 1$ and sufficiently small $\lambda_{0} = \lambda_0(K_0) > 0$ such that, for any $\lambda \in (0,\lambda_0)$, the weak solution $v_{\lambda, K_0}$ of problem \eqref{P-truncado} satisfies 
$$
\|v_{\lambda, K_0}\|_{\infty} \leq K_0.
$$
As a consequence  and, consequently, $v_{\lambda, K_0}$ is a weak solution of problem \eqref{P}.
\end{lemma}

\begin{proof}
Let $L > 0$ and $v_L = \min\{v_{\lambda, K}, L\} \geq 0$. We take $\beta > 1$ and $w_L = v_{\lambda, K} v_L^{\beta -1}$. The function
$$
\varphi := v_{\lambda, K}v_L^{p(\beta - 1)} 
$$
is in $E$, since 
\begin{align*}
    \int_{\mathbb{R}^N} |\nabla \varphi|^p &= \int_{\mathbb{R}^N} \left|\left(1+p(\beta -1) \chi_{\{v_{\lambda, K} < L\}}\right)v_L^{p(\beta -1)}\nabla v_{\lambda, K} \right|^p\\ 
    &\leq |1+p(\beta-1)|^p L^{p^2(\beta-1)} \int_{\mathbb{R}^N} |\nabla v_{\lambda, K}|^p \\ 
    &< +\infty
\end{align*}
and, similarly,
\begin{align*}
    \int_{\mathbb{R}^N} |\nabla \varphi|^q \leq |1+p(\beta-1)|^q L^{pq(\beta-1)} \int_{\mathbb{R}^N} |\nabla v_{\lambda, K}|^q < +\infty.
\end{align*}
Thus, 
\begin{align*}
     \int_{\mathbb{R}^N} \phi_{\lambda, K}(v_{\lambda, K})\varphi &= \int_{\mathbb{R}^N}\left(|\nabla v_{\lambda, K}|^{p-2}\nabla v_{\lambda, K}+|\nabla v_{\lambda, K}|^{q-2}\nabla v_{\lambda, K} \right)\nabla\varphi \\
     &= \int_{\mathbb{R}^N}\left(|\nabla v_{\lambda, K}|^{p-2}\nabla v_{\lambda, K}+|\nabla v_{\lambda, K}|^{q-2}\nabla v_{\lambda, K} \right)\left(v_L^{p(\beta-1)} \nabla v_{\lambda, K} + v_{\lambda, K} \nabla \left(v_L^{p(\beta - 1)} \right)\right) \\
     &= \int_{\mathbb{R}^N} v_L^{p(\beta - 1)}\left(|\nabla v_{\lambda, K}|^{p}+|\nabla v_{\lambda, K}|^{q} \right) + \\
     &+p(\beta - 1)\int_{\mathbb{R}^N} \chi_{\{v_{\lambda, K} < L\}} v_L^{p(\beta - 1)}\left(|\nabla v_{\lambda, K}|^{p}+|\nabla v_{\lambda, K}|^{q} \right) \\
     & \geq \int_{\mathbb{R}^N} v_L^{p(\beta - 1)}\left(|\nabla v_{\lambda, K}|^{p}+|\nabla v_{\lambda, K}|^{q} \right) \\
     &\geq \int_{\mathbb{R}^N} v_L^{p(\beta - 1)}|\nabla v_{\lambda, K}|^{p}.
\end{align*}
On the other hand,
\begin{align*}
    \int_{\mathbb{R}^N} \phi_{\lambda, K}(v_{\lambda, K})\varphi &= \int_{\mathbb{R}^N} \phi_{\lambda, K}(v_{\lambda, K})v_{\lambda, K} v_L^{p(\beta-1)} \leq C_{\lambda, K}\int_{\mathbb{R}^N} v_{\lambda, K}^rv_L^{p(\beta -1)}, 
\end{align*}
where $C_{\lambda, K} := 1 + \lambda K^{s-r} > 0$. Therefore,
\begin{equation*}
    \int_{\mathbb{R}^N} v_L^{p(\beta - 1)}|\nabla v_{\lambda, K}|^{p} \leq C_{\lambda, K}\int_{\mathbb{R}^N} v_{\lambda, K}^rv_L^{p(\beta -1)}. 
\end{equation*}
Now, note that 
\begin{align*}
    \|w_L\|_{p^*}^{p} &\leq S_{*}^{-1} \int_{\mathbb{R}^N} |\nabla w_L|^p = S_{*}^{-1} \int_{\mathbb{R}^N} \left|\left(v_L^{\beta-1} \nabla v_{\lambda, K}\right)+\left((\beta-1)v_{\lambda, K}v_L^{\beta-2} \nabla v_L\right)\right|^p \\
    &\leq  S_{*}^{-1} 2^{p-1}\int_{\mathbb{R}^N}\left(v_L^{p(\beta-1)}|\nabla v_{\lambda, K}|^p + (\beta -1)^p v_L^{p(\beta-1)}\chi_{\{v_{\lambda, K} < L\}}|\nabla v_{\lambda, K}|^p \right) \\
    &=S_{*}^{-1} 2^{p-1}\int_{\mathbb{R}^N}\left(1 + (\beta-1)^p \chi_{\{v_{\lambda, K}<L\}}^{p}\right)v_L^{p(\beta-1)}|\nabla v_{\lambda, K}|^{p} \\
    &\leq S_{*}^{-1} 2^{p-1}\left(1+(\beta-1)^p\right) \int_{\mathbb{R}^N} v_L^{p(\beta-1)}|\nabla v_{\lambda, K}|^{p} \\
    &\leq S_{*}^{-1} 2^{p-1}\left(1+(\beta-1)^p\right) C_{\lambda, K}\int_{\mathbb{R}^N} v_{\lambda, K}^rv_L^{p(\beta -1)} \\
    &= S_{*}^{-1} 2^{p-1}\left(\frac{1}{\beta^p}+\left(1-\frac{1}{\beta}\right)^p\right)\beta^{p} C_{\lambda, K}\int_{\mathbb{R}^N} v_{\lambda, K}^rv_L^{p(\beta -1)} \\
    &\leq  S_{*}^{-1} 2^{p} C_{\lambda, K}\beta^p\int_{\mathbb{R}^N} v_{\lambda, K}^{r-p}(v_{\lambda, K}v_L^{(\beta -1)})^p \\
    &=C_1C_{\lambda, K}\beta^p\int_{\mathbb{R}^N} v_{\lambda, K}^{r-p}w_L^p,
\end{align*}
where $C_1 := S_{*}^{-1} 2^{p}$. Using the H\"older inequality, in view of Lemma \ref{ltd}, the convergence $v_n \rightharpoonup v_{\lambda, K}$ in $E$ and the embedding $E \hookrightarrow L^{p^*}(\mathbb{R}^N)$, we obtain 
\begin{align*}
    \|w_L\|_{p^*}^{p} &\leq C_1 C_{\lambda, K}\beta^p \left(\int_{\mathbb{R}^N}v_{\lambda, K}^{p^*} \right)^{\frac{r-p}{p^*}}\left(\int_{\mathbb{R}^N} w_L^{\alpha^*} \right)^{\frac{p}{\alpha*}} \\
    &= C_1 C_{\lambda, K}\beta^p\|v_{\lambda, K}\|_{p^*}^{r-p}\|w_L\|_{\alpha^*}^{p} \\
    &\leq C_2 C_{\lambda, K}\beta^p\|w_L\|_{\alpha^*}^{p}
\end{align*}
where 
$$
\alpha^* := \frac{pp^*}{p^*-(r-p)}
$$
and $C_2 := C_1C^{r-p}$ 
which does not depends on $\lambda$ and $K$.
If $v_{\lambda, K}^{\beta} \in L^{\alpha^*}(\mathbb{R}^N)$, then 
\begin{align*}
    \|w_L\|_{p^*}^{p} &\leq C_2 C_{\lambda,K}\beta^p \left(\int_{\mathbb{R}^N}(v_{\lambda, K}v_L^{\beta-1})^{\alpha^*} \right)^{\frac{p}{\alpha^*}} \\
    &\leq C_2 C_{\lambda,K}\beta^p \left(\int_{\mathbb{R}^N}v_{\lambda, K}^{\beta\alpha^*} \right)^{\frac{p}{\alpha^*}} \\
    &=C_2 C_{\lambda,K}\beta^p \|v_{\lambda, K}\|_{\beta\alpha^*}^{p\beta} \\
    &< +\infty.
\end{align*}
Applying Fatou's Lemma and taking the limit as $L \to \infty$, we obtain
\begin{align} \label{betagenerico}
   \nonumber \|v_{\lambda, K}\|_{p^*\beta} &= \left(\int_{\mathbb{R}^N} \liminf_{L \to \infty}{|v_{\lambda, K}v_L^{\beta-1}|^{p^*}} \right)^{\frac{1}{p^*\beta}} \leq \liminf_{L \to \infty}{\left(\int_{\mathbb{R}^N}|w_L|^{p^*}\right)^{\frac{1}{p^*\beta}}} = \liminf_{L \to \infty} (\|w_L\|_{p^*}^{p})^{\frac{1}{p\beta}} \\
    &\leq (C_2 C_{\lambda,K})^{\frac{1}{p\beta}}\beta^{\frac{1}{\beta}} \|v_{\lambda, K}\|_{\beta\alpha^*},
\end{align}
if $v_{\lambda, K}^{\beta \alpha^*} \in L^{1}(\mathbb{R}^N)$.

Choosing 
$$
\beta = p^{*} / \alpha^{*} = \frac{p^*-(r-p)}{p} = 1 + \frac{p^*-r}{p} > 1,
$$
by \eqref{betagenerico} we find that
\begin{align} \label{beta1}
    \|v_{\lambda, K}\|_{p^*\beta} \leq (C_2 C_{\lambda,K})^{\frac{1}{p\beta}}\beta^{\frac{1}{\beta}} \|v_{\lambda, K}\|_{p^*}.
\end{align}
In particular, $v_{\lambda, K}^{p^* \beta} \in L^1(\mathbb{R}^N)$. Thus, replacing $\beta$ by $\beta^2 > 1$ in \eqref{betagenerico}, observing $\beta^2 \alpha^* = p^* \beta $ and applying \eqref{beta1}, 
\begin{align*}
    \|v_{\lambda, K}\|_{\beta^2p^*} &\leq (C_2 C_{\lambda,K})^{\frac{1}{p\beta^2}}\beta^{\frac{2}{\beta^2}} \|v_{\lambda, K}\|_{p^* \beta} \leq (C_2 C_{\lambda,K})^{\frac{1}{p}\left(\frac{1}{\beta}+\frac{1}{\beta^2}\right)}\beta^{\frac{1}{\beta}+\frac{2}{\beta^2}} \|v_{\lambda, K}\|_{p^*}.
\end{align*}
Continuing the inductive process, we see that 
\begin{align*}
    \|v_{\lambda, K}\|_{\beta^mp^*} \leq (C_2 C_{\lambda,K})^{\frac{1}{p}\left(\sum_{j=1}^{m}\frac{1}{\beta^m}\right)}\beta^{\sum_{j=1}^{m}\frac{j}{\beta^j}} \|v_{\lambda, K}\|_{p^*},
\end{align*}
for any $m \in \mathbb{N}$. Taking $m \to \infty$, we have
\begin{align*}
    \|v_{\lambda, K}\|_{\infty} \leq (C_2 C_{\lambda,K})^{\gamma_1}\beta^{\gamma_2} \|v_{\lambda, K}\|_{p^*} \leq (C_2 C_{\lambda,K})^{\gamma_1}\beta^{\gamma_2} C_3,
\end{align*}
where 
\begin{equation*}
    \|v_{\lambda, K}\|_{p^*} \leq C_3, \quad \gamma_1 := \frac{1}{p}\sum_{j=1}^{\infty} \frac{1}{\beta^j} < \infty \quad \text{and} \quad \gamma_2 := \sum_{j=1}^{\infty} \frac{j}{\beta^j} < \infty.
\end{equation*}
Finally, we need that 
$$
(C_2 C_{\lambda,K})^{\gamma_1}\beta^{\gamma_2} C_3 \leq K,
$$
that is 
$$
(1+\lambda K^{s-r})^{\gamma_1} \leq \frac{K}{C_2^{\gamma_1}\beta^{\gamma_2}C_3}
$$
or, equivalently, 
$$
\lambda \leq \left(\frac{K^\frac{1}{\gamma_1}}{C_2\beta^{\frac{\gamma_1}{\gamma_2}}C_3^{\frac{1}{\gamma_1}}}-1 \right)\frac{1}{K^{s-r}}.
$$
Fix a sufficiently large $K_0>1$ so that
$$
\frac{K_0^\frac{1}{\gamma_1}}{C_2\beta^{\frac{\gamma_1}{\gamma_2}}C_3^{\frac{1}{\gamma_1}}} > 1
$$
and a sufficiently small $\lambda_0 > 0$ such that 
$$
\lambda_0 \leq \left(\frac{K_0^\frac{1}{\gamma_1}}{C_2\beta^{\frac{\gamma_1}{\gamma_2}}C_3^{\frac{1}{\gamma_1}}}-1 \right)\frac{1}{K_0^{s-r}} \quad \text{and} \quad \lambda_0 \leq M_{K_0}.
$$
Thus, for each $0 < \lambda < \lambda_0$, we get $\|v_{\lambda, K_0}\|_{\infty} \leq K_0$ and this completes the proof.
\end{proof}

\subsection{Proof of Theorem~\ref{T2} (\ref{rmaior})}
We start by showing a Mountain Pass Geometry.
\begin{lemma}
    If $q^* < r,s < p^*$ and $r>\max\{s,p\}$, for any $\lambda<0$, the $I_{\lambda}$ satisfies Mountain Pass Geometry.
\end{lemma}

\begin{proof}
    Indeed, for $\|u\|_E = \rho < 1$, we have 
    $$
    I_{\lambda}(u) \geq \frac{1}{p} \|u\|_E^p - \frac{1}{r} \int_{\mathbb{R}^N} |u|^{r} \geq  \frac{1}{p} \|u\|_E^p - \frac{C_r}{r} \|u\|_E^r = \rho^p\left(\frac{1}{p} - \frac{C_r}{r}\rho^{r-p} \right) := \alpha.
    $$
    Take 
    $$
    0< \rho < \min\left\{1, \left(\frac{r}{pC_r}\right)^{1/(r-p)} \right\} 
    $$
    and, consequently, $\alpha > 0$.
    
    On the other hand, for $\varphi \in E\setminus\{0\}$ such that $\varphi_{+} \neq 0$ in $E$ and all $t>0$,
    $$
    I_{\lambda}(t\varphi) = t^r \left(\frac{t^{p-r}}{p}\int_{\mathbb{R}^N}|\nabla \varphi|^p+\frac{t^{q-r}}{q}\int_{\mathbb{R}^N}|\nabla \varphi|^q-\frac{1}{r}\int_{\mathbb{R}^N} \varphi_{+}^r-\frac{\lambda t^{s-r}}{s}\int_{\mathbb{R}^N}\varphi_{+}^s\right) \overset{t \to +\infty}{\to} -\infty
    $$
    and the proof is complete.
\end{proof}

As usual the lemma above implies that there exists a sequence $\{u_n\} \subset E$ such that 
$$
I_{\lambda}(u_n) = c_{\lambda} + o_n(1) \quad \text{and} \quad \|I'_{\lambda}(u_n)\|=o_n(1),
$$
where $c_{\lambda}$ is the minimax level.

\begin{lemma} \label{ltd3}
    The sequence $\{u_n\}$ is bounded.
\end{lemma}

\begin{proof}
    Indeed, 
    \begin{align*}
        c_{\lambda} + o_n(1) + o_n(1) \|u_n\|_E &\geq I_{\lambda}(u_n) - \frac{1}{r} I'_{\lambda}(u_n)(u_n) \\
        &=\frac{r-p}{pr}\int_{\mathbb{R}^N}|\nabla u_n|^p+ \frac{r-q}{qr}\int_{\mathbb{R}^N}|\nabla u_n|^q+ \lambda \frac{s-r}{rs}\int_{\mathbb{R}^N}u_{n+}^s \\
        &\geq \frac{r-p}{pr}\int_{\mathbb{R}^N}|\nabla u_n|^p+\frac{r-q}{rq}\int_{\mathbb{R}^N}|\nabla u_n|^q\\
        &\geq \frac{r-p}{rp}\kappa\|u_n\|_E - 2\frac{r-p}{rp}.
    \end{align*}
    Therefore, 
    \begin{equation*}
        c_{\lambda} + o_n(1) + 2\frac{r-p}{rp} \geq \left(\frac{r-p}{rp} \kappa - o_n(1) \right) \|u_n\|_E
    \end{equation*}
    and, consequently, the result follows.
\end{proof}

\begin{lemma}
    There exists $C_{\lambda}>0$ such that for all sufficiently large $n \in \mathbb{N}$, we have
    $$
    \int_{\mathbb{R}^N} |u_n|^r \geq C_{\lambda}.
    $$
\end{lemma}

\begin{proof}
Indeed, using that $q<p,r,s$ and $\lambda\leq 0$ we get 
    \begin{align*}
        c_\lambda+o_n(1)&=I_\lambda(u_n)-\frac{1}{q}I'_\lambda(u_n)u_n \\
&=\frac{q-p}{pq}\int_{\mathbb{R}^N}|\nabla u_n|^p-\frac{q-r}{rq}\int_{\mathbb{R}^N}u_{n+}^r-\lambda \frac{q-s}{sq}\int_{\mathbb{R}^N}u_{n+}^s\\
&\leq \frac{r-q}{qr}\int_{\mathbb{R}^N}|u_n|^r,
    \end{align*}
    proving the lemma.
\end{proof}

By the previous lemma, it follows that $u_n \not\to 0$ in $L^{r}(\mathbb{R}^N)$. Moreover, by Lions Lemma, Lemma \ref{ltd3} and arguing as before, there exist $\{v_n\} \subset E$ and $v \in E\setminus\{0\}$ such that $v_n \rightharpoonup v\ge 0$ in $E$, and $ v$ is a solution of the problem.

\subsection{Proof of Theorem~\ref{T2} (\ref{rnomeio})}

Also in this case we have a Mountain Pass Geometry, at least for small $\lambda$.
\begin{lemma} \label{GPM-negativo}
    If $p<r<s$, there are $\alpha_0, \rho_0 > 0, e \in E$  and $\lambda_1<0$ such that 
    \begin{equation*}
        \|e\|_E > \rho_0 \quad \text{and} \quad \inf_{\|u\|_E = \rho_0} J_{\lambda,L}(u) \geq \alpha_0 > J_{\lambda,L}(e),
    \end{equation*}
    for all $\lambda \in (\lambda_1,0)$ and $L>0$.
\end{lemma}
\begin{proof}
    Indeed, if $\|u\|_E  < 1$, then 
    \begin{equation*}
        I_{\lambda}(u) \geq J_{\lambda, L}(u) \geq I_{0}(u) \geq \frac{1}{p}\|u\|_E^{p} - \frac{C_r}{r}\|u\|_E^{r} 
    \end{equation*}
   from which the conclusion follows for suitable $\rho_0\in(0,1)$ and $\alpha_0>0$.
   
    Furthermore, for a  fixed $\varphi \in E\setminus\{0\}$ such that $\varphi\ge 0 $, we have
    \begin{equation*}
        I_{0}(tu_0) = \frac{t^p}{p} \int_{\mathbb{R}^N} |\nabla \varphi|^p + \frac{t^q}{q} \int_{\mathbb{R}^N} |\nabla \varphi|^q - \frac{t^r}{r} \int_{\mathbb{R}^N} \varphi^r \to -\infty
    \end{equation*}
    as $t \to + \infty$,
    implying that  there is $e\in E$ with $\| e\|>\rho_0$
    and $I_0(e)<0$.
 Since
    \begin{equation*}
        \left|I_{\lambda}(e)-I_{0}(e) \right| = -\frac{\lambda}{s} \int_{\mathbb{R}^N}e_{+}^s \to 0
        \quad \text{as }\lambda \to 0,
    \end{equation*}
    there is $\lambda_1 < 0$ sufficiently close to $0$ such that $I_{\lambda}(e)<0$ for all $\lambda \in [\lambda_1,0)$. Consequently, we have 
    $
        J_{\lambda, L}(e) \leq I_{\lambda}(e) < 0,
    $
    for all $\lambda \in [\lambda_1, 0)$ and $L>0$.
\end{proof}

It follows from Lemma \ref{GPM-negativo} that for any $\lambda \in [\lambda_1, 0), L>0$, there exists a sequence  $\{u_n\} := \{u_n^{\lambda, L}\} \subset E$ such that
\begin{equation} \label{J-sequence}
    J_{\lambda,L}(u_n) = c_{\lambda,L} + o_n(1) \quad \text{and} \quad \|J'_{\lambda,L}(u_n)\| = o_n(1)
\end{equation}
where
$$
c_{\lambda, L}: = \inf_{\gamma \in \Gamma} \max_{t \in [0,1]} J_{\lambda, L}(\gamma(t)),
\quad \Gamma := \left\{\gamma \in C\left([0,1]; E \right): \gamma(0) = 0, \ \gamma(1) = e \right\},
$$
and
\begin{equation} \label{minimaxs}
    0 < c_{0} := \inf_{\gamma \in \Gamma} \max_{t \in [0,1]} I_0(\gamma(t)) \leq c_{\lambda, L}  \leq c_{\lambda} = \inf_{\gamma \in \Gamma} \max_{t \in [0,1]} I_{\lambda}(\gamma(t)) < c_{\lambda_{1}} < +\infty,
\end{equation}
Moreover $c_\lambda\to c_0$ as $\lambda\to 0.$

\begin{lemma} \label{ltd-truncado}
    There exists $L_0>0$  and $\lambda_0 = \lambda_0(L_0) \in [\lambda_1, 0)$ such that,
    for any $L\ge L_0$ and any $\lambda\in(\lambda_0, 0)$
      the sequence in \eqref{J-sequence}
      for the functional $J_{\lambda,L}$ is bounded by $L$.
\end{lemma}

\begin{proof}
We argue by contraddition.
Suppose  that for any $L>0$
and $\lambda<0$ near zero,
the sequence $\{u_n\}$ associated to $J_{\lambda,L}$ satisfying \eqref{J-sequence} is such that 
$\|u_n\|_E > L$.
In particular  for
the pair $(L,\lambda)$ where
    $L$ is so large that
\begin{equation}\label{eq:L}
    L> \frac{pr}{\kappa (r-p)}(c_{\lambda_1} + 2\frac{r-p}{pr} + 1)
\end{equation}
 in such a way that $\lambda$ given by
\begin{equation}\label{eq:lambda}
    \lambda =- \left(\frac{2^pp +s-r }{sr}  C_s^s (2L)^s\right)^{-1}
\end{equation}
 belongs to $[\lambda_1,0)$,
 it happens that 
 the sequence $\{u_n\}$ satisfying \eqref{J-sequence} is such that 
$\|u_n\|_E > L$.

If it were
$\|u_n\|_E \geq 2L$ for all $n \in \mathbb{N}$ sufficiently large,
then, since the truncated term is zero,
\begin{align*}
    c_{\lambda_1} + o_n(1) + o_n(1)\|u_n\|_E &\geq c_{\lambda, L} + o_n(1) + o_n(1)\|u_n\|_E = J_{\lambda, L}(u_n) - \frac{1}{r} J'_{\lambda, L}(u_n)[u_n] \\
    &= \frac{r-p}{pr}\int_{\mathbb{R}^N}|\nabla u|^p +  \frac{r-q}{rq} \int_{\mathbb{R}^N}|\nabla u|^q \\
    &\geq \frac{r-p}{pr} \left(\kappa\|u_n\|_E-2 \right). 
\end{align*}
Consequently, for all $n \in \mathbb{N}$ large, 
\begin{eqnarray*}
   c_{\lambda_1} + 2\frac{r-p}{pr} + 1
   &\ge& c_{\lambda} + 2\frac{r-p}{pr} + o_n(1) \\ &\geq& \left(\kappa\frac{r-p}{pr} - o_n(1)\right) \|u_n\|_E \\
    &\geq &\left(\kappa\frac{r-p}{pr} - o_n(1)\right) 2L\ge \kappa \frac{r-p}{pr}L,  
\end{eqnarray*}
 which contradicts \eqref{eq:L}. So it has to be  $L < \|u_n\|_E < 2L$. But then again, 
\begin{align*}
    c_{\lambda_1} + o_n(1) + o_n(1)\|u_n\|_E &= c_{\lambda, L} + o_n(1) + o_n(1)\|u_n\|_E 
    \geq J_{\lambda, L}(u_n) - \frac{1}{r} J'_{\lambda, L}(u_n)[u_n] \\
    &= \frac{r-p}{pr} \int_{\mathbb{R}^N}|\nabla u|^p +  \frac{r-q}{qr} \int_{\mathbb{R}^N}|\nabla u|^q + \lambda \left(\frac{1}{r}H'_L(u_n)u_n - H_L(u_n) \right) \\
    &\ge \frac{r-p}{pr}\left(\kappa\|u_n\|_E-2 \right) +\\
    & \ \ + \lambda \left(\frac{p}{srL^p} \Psi'\left(\frac{\|u_n\|_E^p}{L^p}\right)\|u_n\|_{E}^{p} + \frac{s-r}{rs}\Psi_L(u_n) \right) \int_{\mathbb{R}^N}u_{n+}^s \\
    &\geq  \frac{r-p}{pr}\left(\kappa\|u_n\|_E-2 \right) + \lambda \left(\frac{p}{srL^p} (2L)^p + \frac{s-r}{sr}\right)C_s^s \|u_n\|_E^s, 
\end{align*}
where $C_s > 0$ is the embedding constant of $E$ in $L^s(\mathbb{R}^N)$. Thus, for all $n \in \mathbb{N}$ sufficiently large, we obtain 
\begin{align*}
    c_{\lambda_1} + 2\frac{r-p}{pr} + o_n(1) &\geq \left( \kappa\frac{r-p}{pr} - o_n(1) \right) \|u_n\|_E + \lambda \left(\frac{p}{srL^p} (2L)^p + \frac{s-r}{sr}\right) C_s^s \|u_n\|_E^{s} \\
    &\geq \left( \kappa\frac{r-p}{pr} - o_n(1) \right) L + \lambda \frac{2^pp +s-r }{sr}  C_s^s (2L)^s.
\end{align*}
Consequently, by \eqref{eq:lambda}
we get
\begin{equation*}
    c_{\lambda_1} + 2\frac{r-p}{pr} \geq \kappa\frac{r-p}{pr} L + \lambda \frac{2^pp +s-r }{sr}  C_s^s (2L)^s=
    \kappa\frac{r-p}{pr} L-1,
\end{equation*}
that again contradicts \eqref{eq:L}.
\end{proof}
From now on we assume that
$L\ge L_0$ and $\lambda \in (\lambda_0, 0)$ so that, the functional $J_{\lambda,L}$ which has the Mountain Pass Geometry (uniformly in $L$ and $\lambda$) possesses at the level $c_{\lambda,L}$ a  Palais-Smale sequence which is bounded by $L$ and then this sequence satisfies
\begin{equation*}\label{eq:PS}
     I_{\lambda}(u_n) = c_{\lambda} + o_n(1), \quad  \quad I'_{\lambda}(u_n) = o_n(1)
     \quad \text{and} \quad \|u_n\|\le L.
\end{equation*}

We show now that the sequence $\{u_n\}$
does not vanish.
\begin{lemma}
   We have 
    \begin{equation*}
        \liminf_{n \to \infty}{\int_{\mathbb{R}^N} |u_n|^r} \geq c_0\frac{qr}{r-q}> 0,
    \end{equation*}
    where $c_0 > 0$ is the minimax level in \eqref{minimaxs}.
\end{lemma}

\begin{proof}

   We have
    \begin{align*}
        c_0 + o_n(1) &\leq c_{\lambda} + o_n(1) = I_{\lambda}(u_n) -\frac{1}{q} I'_{\lambda}(u_n)[u_n] \\
        &= \frac{q-p}{pq}\int_{\mathbb{R}^N}|\nabla u_n|^p+\frac{r-q}{qr}\int_{\mathbb{R}^N}u_{n+}^r+\lambda\frac{s-q}{qs}\int_{\mathbb{R}^N}u_{n+}^s \\
        &\le \frac{r-q}{qr}\int_{\mathbb{R}^N}|u_n|^r,
    \end{align*}
    and the result follows.
\end{proof}

Thus, the Lions Lemma tells us that
\begin{equation*}
   \sup_{y\in \mathbb{R}^N}\int_{B_1(y)}|u_n|^t > C_t>0
\end{equation*}
and there exists a sequence $\{y_n\} \subset \mathbb{R}^N$ such that 
\begin{equation*}
    \int_{B_1(y_n)}|u_n|^t \geq C_t.
\end{equation*}
Then the translated sequence
$v_n(x):= u_n(x + y_n)$
 has norm less then $L$, is still 
a Palais-Smale sequence for $I_\lambda$ at the level $c_\lambda$, and weakly converges to a non zero function $v\in E$ with $\|v\|\le L$.

As before $v$ lies within the region where the truncation is inactive, and hence it is also a critical point of $I_\lambda$.

\section{Regularity and positivity of the solutions}
\label{sec:RegPosit}

In this section, we establish regularity properties for weak solutions of problem \eqref{P}. First, we observe that the construction presented in  \cite[proof of Theorem 2, Section 4]{HL08}  can be adapted to our setting without essential modifications. Indeed, by carefully verifying each step of the argument, we note that the local estimates and truncation procedures remain valid for the nonlinearity of the form
$$
f(x,u)=|u|^{r-2}u+\lambda |u|^{s-2}u, \quad \text{with} \quad q^* < r , s < p^*,
$$
considered here. In particular, the structural assumptions required in the cited article are satisfied in our case.

As a direct consequence of this adaptation, we obtain that every weak solution \(u \in E\) of \eqref{P} satisfies
$$
u \in L_{\mathrm{loc}}^{\infty}(\mathbb{R}^N).
$$

Once local boundedness is established, we can apply \cite[Theorem 1]{HL08}, which provides regularity for weak solutions of elliptic equations of the type
$$
\begin{cases}
-\Delta_p v - \Delta_q v = f(x), & x \in \mathbb{R}^N, \\
v \in W^{1,p}_{\mathrm{loc}}(\mathbb{R}^N), & 1 < q < p.
\end{cases}
$$
In our case, the function \(f(x)=|u(x)|^{r-2}u(x)+\lambda |u(x)|^{s-2}u(x)\), where $u \in E$ is a weak solution of the \eqref{P}, satisfies the required conditions of that result, in particular the growth and local continuity assumptions, since \(u\) is locally bounded.
Therefore, we conclude that
$
u \in C^{1,\alpha}_{\mathrm{loc}}(\mathbb{R}^N),
$
for some $\alpha \in (0,1)$. Moreover, in the supercritical case treated via truncation, the obtained solution is, by construction, already bounded. Indeed, considering the associated truncated problem
$$
-\Delta_p u - \Delta_q u = \phi_K(u),
$$
where \(\phi_K\) is defined as in the previous construction, the solution coincides with that of the original problem for sufficiently large values of \(K\), and automatically satisfies an \(L^\infty\) bound. Hence, in this case as well, the assumptions of \cite[Theorem 1]{HL08} remain valid, yielding again the regularity
$
u \in C^{1,\alpha}_{\mathrm{loc}}(\mathbb{R}^N).
$
In conclusion, in all cases considered, weak solutions of \eqref{P} belong to 
$C^{1,\alpha}_{\mathrm{loc}}(\mathbb{R}^N)$. In addition, an application of \cite[Theorem 1.1]{T67} shows that the nonnegative solutions obtained are, in fact, strictly positive.

\section{Nonexistence result}\label{sec:NE}
Building upon the ideas in \cite{K93,PS86} we prove a Pohozaev-type identity for problem \eqref{P}.

For this purpose, let us take for $t \in \mathbb{R}$
\begin{equation*}
    f(t) = |t|^{r-2}t + \lambda |t|^{s-2}t\quad\mbox{and} 
\quad    F(t) = \int_{0}^{t} f(\tau) d\tau = \frac{|t|^{r}}{r} + \lambda \frac{|t|^{s}}{s}.
\end{equation*}
Furthermore, let $\eta \in C_{0}^{\infty}(\mathbb{R}^N)$ be a radial function such that
\begin{equation*}
    0\leq \eta \leq 1, \quad \eta \equiv 1 \quad \textrm{in} \quad B_{1}(0) \quad \textrm{and} \quad \eta \equiv 0 \quad \textrm{in} \quad \mathbb{R}^N \setminus B_{2}(0).
\end{equation*}
For $R > 0$, set 
\begin{equation*}
    \eta_{R}(x) = \eta\left(\frac{x}{R} \right).
\end{equation*}
Then 
\begin{equation*}
    0\leq \eta_{R} \leq 1, \quad \eta_{R} \equiv 1 \quad \textrm{in} \quad B_{R}(0), \quad  \eta_{R} \equiv 0 \quad \textrm{in} \quad \mathbb{R}^N \setminus B_{2R}(0) \quad \textrm{and} \quad |\nabla\eta_{R}| \leq \frac{\|\nabla \eta\|_{\infty}}{R}.
\end{equation*}

\begin{theorem}[Pohozaev Identity]\label{Pohozaev}
Let $u \in E \cap W^{2,p}_{\mathrm{loc}}(\mathbb{R}^N)$
be a weak solution of problem \eqref{P}. Then the following Pohozaev identity holds:
\begin{equation}\label{Pohozaev1}
\frac{1}{p^*}\int_{\mathbb{R}^N}|\nabla u|^p+\frac{1}{q^*}\int_{\mathbb{R}^N}|\nabla u|^q=\frac{1}{r}\int_{\mathbb{R}^N}|u|^r+ \frac{\lambda}{s}\int_{\mathbb{R}^N}|u|^s.
\end{equation}
\end{theorem}
\begin{proof}
The idea is to consider the test function
\begin{equation*}
    \varphi_{R}(x) := \eta_{R}(x)\left(x \cdot \nabla u(x) \right) \in E
\end{equation*} 
in \eqref{idweak}.
Calculating separately, we obtain
\begin{align*}
    \mathcal{I}_{R}^{1} &:= \int_{\mathbb{R}^N} f(u)\varphi_R \\
    &= \int_{\mathbb{R}^N} \operatorname{div} \left(F(u)\eta_{R}x\right) - F(u)\left(\nabla \eta_{R} \cdot x + N \eta_{R} \right) \\
    &= \int_{B_{2R}(0)} \operatorname{div} \left(F(u)\eta_{R}x\right) - \int_{\mathbb{R}^N} F(u)(\nabla \eta_{R} \cdot x) - N\int_{\mathbb{R}^N} \eta_{R} F(u) \\
    &= \int_{\partial B_{2R}(0)} F(u)\eta_{R}x \cdot \left(\frac{x}{2R} \right) dS_{x} - \int_{\mathbb{R}^N} F(u)(\nabla \eta_{R} \cdot x) - N\int_{\mathbb{R}^N} \eta_{R} F(u) \\
    &= - \int_{\mathbb{R}^N} F(u)(\nabla \eta_{R} \cdot x) - N\int_{\mathbb{R}^N} \eta_{R} F(u) \\
    &=: -\mathcal{I}_{R}^{1,1} - \mathcal{I}_{R}^{1,2},
\end{align*}

In addition, we also have
\begin{align*}
    \mathcal{I}_{R}^{2} := \int_{\mathbb{R}^N} |\nabla u|^{p-2}\nabla u \nabla \varphi_{R} &= \int_{\mathbb{R}^N} |\nabla u|^{p-2} (x \cdot \nabla u)(\nabla u \cdot \nabla \eta_{R}) + \int_{\mathbb{R}^N} \eta_{R}|\nabla u|^{p-2}\nabla u \nabla (x \cdot \nabla u) \\
    &=: \mathcal{I}_{R}^{2,1} + \mathcal{I}_{R}^{2,2},
\end{align*}
and similarly
\begin{align*}
    \mathcal{I}_{R}^{3} := \int_{\mathbb{R}^N} |\nabla u|^{q-2}\nabla u \nabla \varphi_{R} &= \int_{\mathbb{R}^N} |\nabla u|^{q-2} (x \cdot \nabla u)(\nabla u \cdot \nabla \eta_{R}) + \int_{\mathbb{R}^N} \eta_{R}|\nabla u|^{q-2}\nabla u \nabla (x \cdot \nabla u) \\
    &=: \mathcal{I}_{R}^{3,1} + \mathcal{I}_{R}^{3,2}.
\end{align*}
Now, applying $\varphi_{R} \in E$ in \eqref{idweak}, 
by the previous identities we obtain 
\begin{equation} \label{eqweak}
    \mathcal{I}_{R}^{2,1} + \mathcal{I}_{R}^{2,2} + \mathcal{I}_{R}^{3,1} + \mathcal{I}_{R}^{3,2} = - \mathcal{I}_{R}^{1,1} - \mathcal{I}_{R}^{1,2}.
\end{equation}
We see that 
\begin{align*}
    \left|\mathcal{I}_{R}^{1,1} \right| &=\left| \int_{R < |x| < 2R} F(u)(\nabla \eta_{R} \cdot x) \right| \leq \int_{R < |x| < 2R} |F(u)||\nabla \eta_{R}||x| \\
    &\le 2\|\nabla \eta\|_{\infty} \int_{|x|>R} \left( \frac{|u|^{r}}{r} + |\lambda| \frac{|u|^{s}}{s} \right) \overset{R \to + \infty}{\longrightarrow} 0,
\end{align*}
\begin{align*}
    \left|\mathcal{I}_{R}^{2,1} \right| &= \left| \int_{R < |x| < 2R} |\nabla u|^{p-2} (x \cdot \nabla u)(\nabla u \cdot \nabla \eta_{R}) \right| \leq \int_{R < |x| < 2R} |\nabla u|^{p} |x||\nabla \eta_{R}| \\
    &\leq 2\|\nabla \eta\|_{\infty} \int_{|x| > R} |\nabla u|^{p} \overset{R \to + \infty}{\longrightarrow} 0
\end{align*}
and, similarly,
\begin{equation*}
    \left|\mathcal{I}_{R}^{3,1}\right| \overset{R \to + \infty}{\longrightarrow} 0.
\end{equation*}
Note also that it is 
\begin{align*}
    \left|\mathcal{I}_{R}^{1,2} - N\int_{\mathbb{R}^N}F(u)\right| & \leq N \int_{\mathbb{R}^N}(1-\eta_R)|F(u)| \overset{R \to + \infty}{\longrightarrow} 0,
\end{align*}
due to the Lebesgue's Dominated Convergence Theorem. A straightforward calculation shows that 
\begin{align*}
    \mathcal{I}_{R}^{2,2} &= \int_{\mathbb{R}^N}\left[ \eta_{R}|\nabla u|^{p} + \eta_{R}|\nabla u|^{p-2} \sum_{i,j = 1}^{N}x_j u_{x_i}u_{x_i x_j} \right] \\
    &= \int_{\mathbb{R}^N}\left[\eta_{R}|\nabla u|^{p} + \frac{1}{p} \eta_{R} x \cdot \nabla\left(|\nabla u|^{p} \right) \right] \\
    &= \int_{\mathbb{R}^N} \left[\eta_{R}|\nabla u|^{p} + \operatorname{div}\left(\frac{|\nabla u|^{p}}{p}\eta_{R} x \right) - \frac{N}{p}\eta_{R} |\nabla u|^{p}  -  \frac{|\nabla u|^{p}}{p}(x\cdot \nabla \eta_{R}) \right]\\
    &= \int_{B_{2R}(0)} \operatorname{div}\left(\frac{|\nabla u|^{p}}{p}\eta_{R} x \right) - \int_{\mathbb{R}^N}\frac{N-p}{p}\eta_{R} |\nabla u|^{p}  -  \int_{\mathbb{R}^N}\frac{|\nabla u|^{p}}{p}(x\cdot \nabla \eta_{R})\\
    &=-  \frac{N-p}{p} \int_{\mathbb{R}^N}\eta_{R} |\nabla u|^{p}  -  \int_{\mathbb{R}^N} \frac{|\nabla u|^{p}}{p}(x\cdot \nabla \eta_{R}) \\
    &=: - \mathcal{I}_{R}^{2,2,1} -  \mathcal{I}_{R}^{2,2,2}.
\end{align*}
Since
\begin{align*}
    \left|\mathcal{I}_{R}^{2,2,2} \right| &= \left| \int_{R<|x|<2R} \frac{|\nabla u|^{p}}{p}(x\cdot \nabla \eta_{R}) \right| \leq \frac{2\|\nabla \eta\|_{\infty}}{p} \int_{|x|>R} |\nabla u|^p \overset{R \to + \infty}{\longrightarrow} 0
\end{align*}
and 
\begin{align*}
    \left|\mathcal{I}_{R}^{2,2,1} - \int_{\mathbb{R}^N} \frac{N-p}{p} |\nabla u|^{p} \right| & \leq \frac{N-p}{p} \int_{\mathbb{R^N}}(1-\eta_{R})|\nabla u|^p \overset{R \to + \infty}{\longrightarrow} 0,
\end{align*}
it follows that
\begin{equation*}
    \mathcal{I}_{R}^{2,2} \overset{R \to + \infty}{\longrightarrow} - \frac{N-p}{p}\int_{\mathbb{R}^N} |\nabla u|^{p}.
\end{equation*}
Similarly  we have
\begin{equation*}
    \mathcal{I}_{R}^{3,2} \overset{R \to + \infty}{\longrightarrow} - \frac{N-q}{q}\int_{\mathbb{R}^N} |\nabla u|^{q}.
\end{equation*}
Finally, setting $R \to +\infty$ in \eqref{eqweak}, we get
\begin{equation*}
    - \frac{N-p}{p}\int_{\mathbb{R}^N} |\nabla u|^{p} - \frac{N-q}{q}\int_{\mathbb{R}^N} |\nabla u|^{q} = -\frac{N}{r} \int_{\mathbb{R^N}} |u|^{r} - \frac{\lambda N}{q} \int_{\mathbb{R^N}} |u|^{s}
\end{equation*}
and the result follows.
\end{proof}

\subsection{Proof of Theorem \ref{th:NE}}

If $u \in E \cap W^{2,p}_{\mathrm{loc}}(\mathbb{R}^N)$ is a weak solution of \eqref{P}, then by Theorem~\ref{Pohozaev} 
$$
\int_{\mathbb{R}^N} |\nabla u|^p + \int_{\mathbb{R}^N} |\nabla u|^q = \int_{\mathbb{R}^N}|u|^r + \lambda\int_{\mathbb{R}^N}|u|^s.
$$
This, together with \eqref{Pohozaev1} implies the identities
\begin{equation*} \label{id1}
    \frac{p^*-q^*}{p^*q^*} \int_{\mathbb{R}^N} |\nabla u|^{q} = \frac{p^*-r}{p^*r}\int_{\mathbb{R}^N}|u|^{r} + \lambda \frac{p^*-s}{p^*s} \int_{\mathbb{R}^N}|u|^{s}
\end{equation*}
and
\begin{equation*} \label{id2}
    \frac{q^*-p^*}{p^*q^*} \int_{\mathbb{R}^N} |\nabla u|^{p} = \frac{q^*-r}{q^*r} \int_{\mathbb{R}^N} |u|^{r} + \lambda \frac{q^*-s}{q^*s} \int_{\mathbb{R}^N}|u|^{s},
\end{equation*}
from which all the statements of Theorem~\ref{th:NE} follow immediately.

\subsection*{Acknowledgements}

E. S. Medeiros was supported by 
CNPQ/Grants 310885/2023-0 (Brazil).

A. A. Nascimento was supported by 
Capes - 88887.949291/2024-00 (Brazil)
and PDSE/Capes program 88881.219456/2025-01 (Brazil).

G. Siciliano  was partially supported
by  
GNAMPA-INdAM Project 2025, CUP
E5324001950001 (Italy),
Capes, CNPq projects 402514/2024-6 and 304244/2023-6, Fapesp projects 2022/16407-1
and 2022/16097-2 (Brazil).

\medskip

\noindent {\bf Data availability statement} Date sharing is not applicable to this article as no new 
data were created or analyzed in this study.

\medskip

\noindent {\bf Competing interests} The authors declare that they have no competing interests.

\medskip

\noindent {\bf Author contributions} All authors contributed equally to the writing of this article. All 
authors read and approved the final manuscript.

\medskip

\noindent {\bf Conflict of interest} The authors declare no conflict of interests.

\end{document}